\documentclass[a4paper,10pt]{article}

\usepackage{fullpage}
\usepackage[T1]{fontenc}
\usepackage{graphicx}
\usepackage{authblk}
\usepackage{amsmath,amsfonts,amsbsy,amssymb,amsthm}
\usepackage{xcolor}
\usepackage{hyperref}
\usepackage[numbers,sort&compress]{natbib}

\newtheorem{theorem}{Theorem}
\newtheorem{lemma}{Lemma}
\newtheorem{proposition}{Proposition}
\newtheorem{corollary}{Corollary}
\theoremstyle{definition}
\newtheorem{remark}{Remark}
\newtheorem{definition}{Definition}

\numberwithin{equation}{section}
\numberwithin{theorem}{section}
\numberwithin{proposition}{section}
\numberwithin{lemma}{section}
\numberwithin{corollary}{section}

\DeclareMathOperator{\Tr}{Tr}

\newcommand{\R}{\mathbb{R}}
\newcommand{\D}{\mathcal{D}}
\newcommand{\F}{\mathcal{F}}
\newcommand{\Rc}{\mathcal{R}}

\newcommand{\E}{\mathbb{E}}

\newcommand{\N}{\mathbb{N}}
\renewcommand{\O}{\mathcal{O}}
\newcommand{\A}{\mathcal{A}}

\newcommand{\bp}{\boldsymbol{p}}
\newcommand{\bs}{\boldsymbol{s}}

\newcommand{\floor}[1]{\lfloor #1 \rfloor}
\renewcommand{\choose}[2]{\begin{pmatrix} #1 \\ #2 \end{pmatrix}}

\renewcommand{\L}{\mathcal{L}}
\newcommand{\Var}{{\rm Var}}
\newcommand{\TT}{\mathfrak{T}}
\newcommand{\RT}{\mathfrak{R}}
\newcommand{\bq}{\boldsymbol{q}}
\newcommand{\e}{\mathbf{e}}

\title{A random walk approach to linear statistics in random tournament ensembles}
\date{}

\author[1]{Christopher H. Joyner}
\author[2]{Uzy Smilansky}

\affil[1]{\normalsize \emph{School of Mathematical Sciences, Queen Mary University of London, London, E1 4NS, United Kingdom}}
\affil[2]{\normalsize \emph{Department of Physics of Complex Systems, Weizmann Institute of Science, Rehovot 7610001, Israel}}

\begin{document}

\maketitle

\begin{abstract}
We investigate the linear statistics of random matrices with purely imaginary Bernoulli entries of the form $H_{pq} = \overline{H}_{qp} = \pm i$, that are either independently distributed or exhibit global correlations imposed by the condition $\sum_{q} H_{pq} = 0$. These are related to ensembles of so-called random tournaments and random regular tournaments respectively. Specifically, we construct a random walk within the space of matrices and show that the induced motion of the first $k$ traces in a Chebyshev basis converges to a suitable Ornstein-Uhlenbeck process. Coupling this with Stein's method allows us to compute the rate of convergence to a Gaussian distribution in the limit of large matrix dimension. 
\end{abstract}


\section{Introduction}
The idea of using a stochastic dynamical evolution to unearth the spectral properties of random matrices was first proposed by Dyson \cite{Dyson-1962}. His insight was that, by initiating a suitable Brownian motion within the space of certain invariant matrix ensembles, one could induce a corresponding motion in the eigenvalues, which is independent of the eigenvectors. Thus, solving the associated Fokker-Planck equation for the stationary solution would recover the joint probability density function for the eigenvalues. Dyson Brownian motion (DBM), as it is now known, has since become a powerful tool in random matrix theory (see for instance \cite{Anderson-2009,Forrester-2010,Erdos-2017}). In \cite{Joyner-2015} the present authors advocated an approach in which the idea of using stochastic dynamics to obtain spectral statistics was extended to Bernoulli matrix ensembles. In particular, we argued heuristically, that by initiating a suitable discrete random walk in the space of matrices, the induced motion of the eigenvalues would tend, in some fashion, to DBM in the limit of large matrix size. Then, as a consequence, the spectral properties of Bernoulli matrices would converge to those of the Gaussian orthogonal ensemble (GOE). In the present article we apply this approach to the linear-statistics of matrices associated to random tournaments and random regular tournaments. Tournament graphs are widely studied objects in combinatorics, with results and open questions regarding, their enumeration, score sequences, cycle properties and Perron-Frobenius eigenvalues for instance \cite{Spencer-1974,McKay-1990,McKay-1996,Gao-2000,Gervacio-1988,Kuhn-2012,Kirkland-2003}. However, beyond \cite{Sosoe-2016}, there appears to be little analysis from a random matrix theory perspective.

For a random (self-adjoint) matrix $M$, the linear-statistic, for some function $h$, refers to the distribution of the following random variable,
\begin{equation}\label{Eqn: Linear statistic}
\Phi_h(M) := \Tr\left(h(M)\right) = \sum_{\mu=1}^{N} h\left(\lambda_\mu(M) \right),
\end{equation}
where $\lambda_\mu(M)$ are the eigenvalues of $M$. For $N \times N$ random matrices $M$ with appropriately scaled and suitably chosen iid elements, Wigner showed \cite{Wigner-1955,Wigner-1958} that as $N \to \infty$ the expectation, for polynomial functions $h$, converges to the semicircle distribution, i.e.
\begin{equation}\label{Eqn: Wigner result}
\frac{1}{N}\E[\Tr\left(h(M)\right)] \to \frac{2}{\pi}\int_{-1}^{1} h(\lambda) \sqrt{1- \lambda^2} d\lambda \qquad N \to \infty .
\end{equation}
In addition, Wigner showed the variance satisfies $\frac{1}{N^2}\Var[h(M)] = \O(N^{-2})$. This result is therefore, in some respects, analogous to the law of large numbers in standard probability theory.

One is therefore led to the question regarding fluctuations about this mean, i.e. what is the distribution of $\Phi_h(M) - \E[\Phi_h(M)]$ for some particular random matrix ensemble?  This was first addressed by Jonsson \cite{Jonsson-1982} in the case of Wishart matrices, showing this random variable is Gaussian in the large $N$ limit. Later, this was also shown to be the case for Wigner matrices \cite{Khorunzhy-1996,Sinai-1998} and also for $\beta$-ensembles with appropriate potentials \cite{Johansson-1998} for various forms of the test function $h$. Notice there is no obvious analogy with the classic CLT, since the eigenvalues in (\ref{Eqn: Linear statistic}) are highly correlated, meaning the usual $1/\sqrt{N}$ normalisation is not required. Proving this behaviour has become a key part of the universality hypothesis within random matrix theory, since it addresses the global spectral behaviour, and has thus garnered much attention since the first results were established. For instance, many authors have attempted to classify for which test functions $h$ the Gaussian behaviour is retained \cite{Bai-2005,Dallaporta-2011,Lytova-2009,Shcherbina-2011,Sosoe-2013}. Others have investigated large deviation aspects \cite{Guionnet-2002}, rates of convergence \cite{Chatterjee-2007} or different kinds of random matrix ensembles such as band matrices \cite{Anderson-2006} or those with non-trivial correlations \cite{Chatterjee-2007,Schenker-2007}.

To show the convergence of (\ref{Eqn: Linear statistic}) for all polynomial test functions of degree $k$ one may, instead, show the joint convergence for a polynomial basis. A particularly convenient choice are the Chebyshev polynomials of the first kind
\begin{equation}\label{Eqn: Chebyshev first kind}
T_n(x) : = \cos(n\arccos(x)) = \sum_{r=0}^{\floor{\frac{n}{2}}} d^{(n)}_{r}x^{n-2r},  \qquad d^{(n)}_r = (-1)^r\frac{n}{2}\frac{(n-r-1)!}{r!(n-2r)!}2^{n-2r}.
\end{equation}
If one takes the traces 
\begin{equation}\label{Eqn: Chebyshev traces}
\Tr(T_n(M)) := \sum_{r=0}^{\floor{\frac{n}{2}}} d^{(n)}_{r}\Tr(M^{n-2r}),
\end{equation}
then it was first observed by Johansson \cite{Johansson-1998} that if $M$ is chosen from one of the standard Gaussian ensembles, then in the limit of large matrix size the random variables $(\Tr(T_1(M)),\ldots,\Tr(T_k(M)))$ converge to independent Gaussian random variables.
 
A Brownian motion approach has already been used to show convergence to independent Gaussian random variables of $\Tr(T_n(M))$ in the Gaussian unitary ensemble \cite{Cabanal-Duvillard-2001} and more general $\beta$-ensembles \cite{Lambert-2017}, as well as traces of unitary matrices $\Tr(U^n)$ in the classical compact groups \cite{Dobler-2011} and the circular $\beta$-ensembles \cite{Webb-2016}. In particular, the works \cite{Lambert-2017,Dobler-2011,Webb-2016} utilised a multivariate form of Stein's method, developed by Chatterjee \& Meckes and Reinert \& R\"{o}llin \cite{Chatterjee-2008,Reinert-2009,Meckes-2009}, to obtain rates of convergence, something which, beyond \cite{Chatterjee-2007}, is often neglected in the analysis of linear statistics. However the scenarios \cite{Cabanal-Duvillard-2001,Dobler-2011,Webb-2016} have involved invariant matrix ensembles, which have permitted the use of exact expressions for the eigenvalue motion, which are not available in this context. We therefore turn to an alternative combinatorial approach, similar to that applied in \cite{Joyner-2017} for the unimodular ensemble and \cite{Dumitriu-2013,Johnson-2015} for random regular graphs. In particular, we express the variables $\Tr(T_n(M))$ in terms of sums over non-backtracking cycles and analyse how these behave under the random walk. The difficulties arise in providing accurate bounds for the remainder terms, which involve the expectations of certain products of matrix elements with respect to the appropriate ensembles.

\vspace{10pt}

The article is outlined as follows: In Section \ref{Sec: Definitions and results} we discuss the ensembles of random tournaments and random regular tournaments, which lead to Definition \ref{Dfn: ITE} and Definition \ref{Dfn: RITE} for the matrix ensembles we call the imaginary tournament ensemble (ITE) and regular imaginary tournament ensemble (RITE) respectively. We then present our main results, given in Theorem \ref{Thm: ITE theorem} and Theorem \ref{Thm: RITE theorem}, which provide rates of convergence to independent Gaussian random variable of the first $k$ traces of Chebyshev polynomials for matrices in the ITE and RITE respectively. In Section \ref{Sec: Random walks} we attempt to give an intuitive explanation of the random walk approach, including Theorem \ref{Thm: Stein theorem} (due to \cite{Chatterjee-2008,Reinert-2009,Meckes-2009}) regarding the multidimensional exchangeable pairs approach to Stein's method, and briefly outline the the methods used to evaluate the appropriate remainder terms.

In Section \ref{Sec: Graph theoretical tools} we introduce some graph theoretical tools required for subsequent analysis. Sections \ref{Sec: ITE} and \ref{Sec: RITE} are dedicated to showing how to construct suitable random walks for the ITE and RITE respectively. Specifically, we prove Propositions \ref{Prop: ITE evolution} and \ref{Prop: RITE evolution} (respectively) which show the remainders contained in Theorem \ref{Thm: Stein theorem} are small enough to allow for the results of Theorem \ref{Thm: ITE theorem} and Theorem \ref{Thm: RITE theorem}. In particular, although interesting in its own right, the ITE will serve as an illustrative example that the approach works in simple settings and will help introduce ideas also needed for the more complicated RITE.

Finally, in Section \ref{Sec: Conclusion} we offer some concluding thoughts and remarks about possible extensions and in the appendix we collect some necessary theorems, proofs and identities. In particular, Appendix \ref{App: Expectation in RITE} contains a proof for the growth rate of expectation of products of matrix elements in the RITE. This is adapted from the work of McKay \cite{McKay-1990} on the number of regular tournaments and is critical in estimating the remainders in Proposition \ref{Prop: RITE evolution}.

\section{Main results}\label{Sec: Results}

\subsection{Definitions and results}\label{Sec: Definitions and results}

A tournament graph on $N$ vertices is a complete graph in which every edge has a specific orientation (see e.g. Figure \ref{Fig: RW for ITE}). Player $p$ is said to win against player $q$ (equivalently player $q$ loses against player $p$) if there is a directed edge from $p$ to $q$. This is represented by an a adjacency matrix $A$ admitting the property that $A_{pq} = 1-A_{qp} = 1$ (resp. $0$) if player $p$ wins (resp. loses). Since a player can't play themselves the diagonal $A_{pp} = 0$. We denote the set of tournaments on $N$ vertices as $\TT_N$, with cardinality $|\TT_N| = 2^{N(N-1)/2}$ - the number of possible choices of direction for each edge.

If all players win the same number of games, or equivalently the number of incoming edges into a vertex is equal to the number of out going edges for every vertex, then the tournament graph is said to be \emph{regular} (see e.g. Figure \ref{Fig: RW for RITE}). This is characterised by the condition $\sum_{q} A_{pq} = (N-1)/2$ for all $p=1,\ldots,N$, which enforces that $N$ must be odd. We denote the set of \emph{regular tournaments} on $N$ vertices by $\RT_N$. An exact formula for $|\RT|$ is not available however McKay showed \cite{McKay-1990} (improving on an earlier estimate of Spencer \citep{Spencer-1974}) that for large $N$ 
\begin{equation}\label{No. RT}
|\RT_N| = \frac{2^{(N^2-1)/2}e^{-1/2}}{\pi^{(N-1)/2}N^{N/2-1}} \left( 1+\mathcal{O}(N^{-1/2+\epsilon})\right).
\end{equation}
In particular, one observes that $|\RT_N|/|\TT_N| \to 0$ as $N \to \infty$ and therefore one cannot immediately infer properties of regular tournaments from tournaments by ergodicity arguments. Hence, the restriction of the rows sums must be dealt with another manner.

Due to the non-symmetric nature of the adjacency matrices the eigenvalues are, in general, complex. However applying the simple transformation $H = i(2A - (\mathbf{E}_N - \mathbf{I}_N))$ (where $i = \sqrt{-1}$ and $\mathbf{E}_N$ is the $N \times N$ matrix in which every element is 1) brings the matrices into a self-adjoint form. Thus $A_{pq} = 0$ (resp. $1$) corresponds to $H_{pq} = +i$ (resp. $-i$) for all off-diagonal elements $p \neq q$ and $H_{pp} = 0$ for all $p = 1,\ldots,N$. Importantly this means taking complex conjugation yields $\overline{H} = - H$, which in turn implies that if $\lambda$ is an eigenvalue of $H$ then so is $-\lambda$, with the eigenvectors being complex conjugates of each other. This spectral symmetry implies
\begin{equation}\label{Eqn: Odd identity}
\Tr(H^n) \equiv 0 \ \ \forall \ n \ \text{odd} \ .
\end{equation}
In order to make a distinction we say that $H$ is an \emph{imaginary tournament matrix} (resp. \emph{regular imaginary tournament matrix}) if $A = \frac{1}{2}(\mathbf{E}_N - \mathbf{I}_N -iH)$ is a tournament (resp. regular tournament). Therefore, with a slight abuse of notation, we write either $H \in \TT_N$ or $H \in \RT_N$ respectively.

\begin{definition}[Imaginary tournament ensemble]\label{Dfn: ITE}Let $\TT_N$ be the set of imaginary tournament matrices of size $N$. Then the imaginary Bernoulli ensemble (ITE) is given by the set of $H \in \TT_N$ with the uniform probability measure $P(H) = |\TT_N|^{-1}$.
\end{definition}
\begin{definition}[Regular imaginary tournament ensemble]\label{Dfn: RITE}Let $\RT_N$ be the set of regular imaginary tournament matrices of size $N$ (with $N$ being odd). Then the random imaginary tournament ensemble (RITE) is given by the set of $H \in \RT_N$ with the uniform probability measure $P(H) = |\RT_N|^{-1}$.
\end{definition}
Note that Definition \ref{Dfn: ITE} is equivalent to choosing the entires $H_{pq}$ equal to $\pm i$ independently and with equal probability, whereas Definition \ref{Dfn: RITE} is equivalent to choosing $H_{pq}$ equal to $\pm i$ with equal probability but with the constraint that $\sum_{q} H_{pq} = 0$ for all $p=1,\ldots,N$. 

Due to the independence of the elements in the ITE, many of the techniques developed to treat Wigner matrices are directly applicable, for example the universality of local statistics has been established in this case \cite{Sosoe-2016}. Moreover, since $H$ is related to $A$ by a (complex) rank one perturbation, the spectral properties of the ITE can be related to the complex eigenvalues of random tournaments \cite{Sosoe-2016}. However, to the best of our knowledge, there are no such results for the RITE, although linear statistics \cite{Dumitriu-2013,Johnson-2015}, local semicircle estimates \cite{Bauerschmidt-2015a,Bauerschmidt-2017} and local universality results \cite{Bauerschmidt-2015b} have been obtained for random regular graphs using switching methods.

\begin{theorem}[Convergence for ITE]\label{Thm: ITE theorem}Let $Z = (Z_2,Z_3,\ldots,Z_k)$ be a collection of iid random Gaussian variables with mean 0 and variance $\sigma_n^2 = \E[Z_n^2] = n$. Let $H$ be chosen according to the ITE and define the random variables
\begin{equation}\label{Eqn: Centred Chebyshev}
Y_n(H):= \Tr\left(T_{2n}\left(\frac{H}{\sqrt{4N}}\right)\right) - \E\left[ \Tr\left(T_{2n}\left(\frac{H}{\sqrt{4N}}\right)\right) \right].
\end{equation}
Then, for $Y(H) = (Y_2(H),Y_3(H)\ldots,Y_k(H))$, $\phi \in C^2(\R^{k-1})$ with $k$ fixed and $N$ sufficiently large
\[
|\E[\phi(Y(H))] - \E[\phi(Z)]| \leq \O(N^{-1})\| \phi \| + \O(N^{-1})\|\nabla \phi \| + \O(N^{-1})\|\nabla^2 \phi \|.
\]
where
\begin{equation}\label{Eqn: Sup deriv norm}
\|\nabla^j\phi \| := \sup_{Q}\max_{n_1,\ldots,n_j}\left| \frac{\partial^j \phi(Q)}{\partial Q_{n_1}\ldots \partial Q_{n_j}}\right|.
\end{equation}
\end{theorem}

\begin{theorem}[Convergence for RITE]\label{Thm: RITE theorem}Let $Z$ and $Y_n(H)$ be as in Theorem \ref{Thm: ITE theorem} and let $H$ be chosen according to the RITE. Then, for $Y(H) = (Y_2(H),Y_3(H)\ldots,Y_k(H))$, $\phi \in C^2(\R^{k-1})$ with $k$ fixed and $N$ sufficiently large. Then
\begin{equation}\label{Eqn: RITE distributional distance}
|\E[\phi(Y(H))] - \E[\phi(Z)]| \leq \O(N^{-1/2})\| \phi \| + \O(N^{-1})\|\nabla \phi \| + \O(N^{-1})\|\nabla^2 \phi \|.
\end{equation}
\end{theorem}

\begin{proof}The proof of Theorem \ref{Thm: ITE theorem} requires incorporating the results of Proposition \ref{Prop: ITE evolution} into Theorem \ref{Thm: Stein theorem}. For Theorem \ref{Thm: RITE theorem} we incorporate the results of Proposition \ref{Prop: RITE evolution} into Theorem \ref{Thm: Stein theorem}.
\end{proof}

Note that we exclude all the odd Chebyshev polynomials since they are comprised entirely of odd traces (see Equation (\ref{Eqn: Chebyshev first kind})) and so by (\ref{Eqn: Odd identity}) they are identically zero. In addition we have $\Tr(H^2) = \sum_{p,q}H_{pq}H_{qp} = N(N-1)$ for all $H$, which means $\Tr[T_2(H/\sqrt{4N})]$ is constant\footnote{One may consider this at odds with the Gaussian case but it was shown in \cite{Maciazek-2015} the first two moments in the Gaussian $\beta$-ensembles can be scaled in such a way that they may be considered independently of all other moments.}

\subsection{Outline of ideas and methods}\label{Sec: Random walks}

In order to prove Theorems  \ref{Thm: ITE theorem} and \ref{Thm: RITE theorem} we introduce random walks within $\TT_N$ and $\RT_N$ with two properties. Firstly, the stationary distributions correspond to $P(H) =  |\TT_N|^{-1}$ and $P(H) = |\RT_N|^{-1}$, as per Definitions \ref{Dfn: ITE} and \ref{Dfn: RITE} respectively. Secondly, the induced motion of the random variable $Y(H)$ will be closely described by a process, whose stationary distribution is given by $Z = (Z_2,Z_3,\ldots,Z_k)$, as in Theorems \ref{Thm: ITE theorem} and \ref{Thm: RITE theorem}.

More precisely, suppose that at some discrete-time $t \in \N$ our random walker is situated at the matrix $H$, then we have a transition probability $\rho(H \to H')$ for the walker to be at the matrix $H'$ at time $t+1$ later. From this one may track how the corresponding variable $Y_n(H)$ changes to $Y_n(H')$. For instance, since this is a Markov process, the expected change is given by
\begin{equation}\label{Eqn: Drift coefficient}
\E[\delta Y_n | H] := \sum_{H'} \rho(H \to H') [Y_n(H') - Y_n(H)].
\end{equation}
Similarly fluctuations are obtained by calculating the second moment
\begin{equation}\label{Eqn: Diffusion coefficient}
\E[\delta Y_n \delta Y_m | H] := \sum_{H'} \rho(H \to H') [Y_n(H') - Y_n(H)][Y_m(H') - Y_m(H)].
\end{equation}
Now suppose that, if we design our random walk correctly, we observe the moments take the form
\begin{eqnarray}
\E[\delta Y_n | H] & =&  \alpha_N[ - n Y_n(H) + R_n(H) ]\\
\E[\delta Y_n \delta Y_m |H] & = & \alpha_N[2n^2\delta_{nm} + R_{nm}(H) ],
\end{eqnarray}
where $\alpha_N$ is a certain constant depending only on $N$ and $R_n(H), R_{nm}(H)$ are small remainders (the nature of small will be clarified later). Then, for arbitrary test functions $f \in C^3(\R^{k-1})$, expanding $f(Y(H')) = f(Y(H) + \delta Y(H,H'))$ in a Taylor series gives
\begin{align}
 \frac{1}{\alpha_N}\E[\delta f| H]
 & =  \frac{1}{\alpha_N}\sum_{n=2}^k \E[\delta Y_n|H]\frac{\partial f}{\partial Y_n} + \frac{1}{2} \sum_{n,m=2}^k \E[\delta Y_n \delta Y_m|H]\frac{\partial^2 f}{\partial Y_n\partial Y_m} + \E[S_f(H,H')|H]  \nonumber \\
 & =  \A f(Y(H))  +  \sum_{n=2}^k R_n(H)\frac{\partial f}{\partial Y_n} +\frac{1}{2} \sum_{n,m=2}^k R_{nm}(H)\frac{\partial^2 f}{\partial Y_n\partial Y_m} + \frac{1}{\alpha_N}\E[S_f(H,H')|H] ,  \label{Eqn: Observable evolution}
\end{align}
with remainder $S_f(H,H')$ and operator $\A$ given by
\begin{equation}\label{Eqn: Stein operator}
\A := \sum_{n=2}^k \left[n^2\frac{\partial^2 }{\partial X_n^2} - nX_n\frac{\partial}{\partial X_n}\right].
\end{equation}
If the Markov process is started from a unique stationary state, then the distributions of $H$ and $H'$ will be the same, in which case the random variables are referred to as an \emph{exchangeable pair}. Moreover, the expected change in $f$ satisfies $\E[\delta f] = \E[f(Y(H'))] - \E[f(Y(H))] = 0$, which, in turn, means $0 = \alpha_N^{-1}\E[\delta f] = \E[ \A f(Y(H))] + \E[\Rc(H)]$, where $\Rc(H)$ denotes the total remainder in (\ref{Eqn: Observable evolution}). The connection with the Gaussian distribution $Z$ now emerges, since if it were the case the remainder $\E[\Rc(H)]$ is equal to 0 for all test functions $f$ then we would have the following result, known as Stein's Lemma.

\begin{lemma}[Stein's Lemma]\label{Lem: Stein's lemma}Let $\A$ be the operator given in (\ref{Eqn: Stein operator}). Then $\E[\A f(Z)] = 0$ for all $f \in C^2(\R^k)$ if and only if $Z = (Z_2,Z_3,\ldots,Z_k)$, where $Z_n \sim N(0,n)$.
\end{lemma}
\begin{proof}One should consult e.g. Lemma 1 in \cite{Meckes-2009} for details. Although briefly - using the stationarity of the solution with respect to $\A^*$ (see Equation (\ref{Eqn: FP eqn})), integration by parts yields $\E[\A f(Z)] := \int dZ \ P(Z) \A f(Z) = \int dZ  \ f(Z) \A^*P(Z) = 0$ for any $f \in C^2(\R^k)$ and thus establishes the first implication. For the converse one requires the exact form of the solution to equation (\ref{Eqn: Stein equation}) presented in Proposition \ref{Prop: Stein sol} in Appendix \ref{Sec: Stein solution}.
\end{proof}

Of course the remainder will not, in general, be zero but one might expect that if it is close (in some appropriate manner) then the corresponding variable $Y(H)$ will be close to $Z$. Stein's realisation was that $\A$ and $f$ could be connected via an auxiliary test function $\phi$ in what is now known as \emph{Stein's equation}
\begin{equation}\label{Eqn: Stein equation}
\A f(x) = \E[\phi(Z)] -  \phi(x),
\end{equation}
with $Z$ as in Lemma \ref{Lem: Stein's lemma}. Thus taking the expectation with respect to $Y(H)$ gives $|\E[\phi(Y)] - \E[\phi(Z)]| = |\E [\A f(Y)] |$. The aim is therefore to find a bound for $|\E [\A f(Y)] |$ using the function $\phi$, as this will allow for an estimate on the distributional distance between $Y$ and $Z$. This idea was initially developed by Charles Stein as an alternative method for proving the classical CLT \cite{Stein-1972}. Stein's method now refers to the overall technique of recovering the distributional distance from bounding the quantity $\E[\A f(Z)]$. For readers unfamiliar with the basics of Stein's method, the review by Ross \cite{Ross-2011} provides an excellent introduction and overview of the different ways this may be achieved.

The work of G\"{o}tze \cite{Gotze-1991} and Barbour \cite{Barbour-1990} in the early 90s allowed for an extension of Stein's method to multivariate Gaussian distributions and established an explicit connection between Stein's method and Markov processes. Using these ideas a number of authors adapted the use of the exchangeable pairs mechanism to multivariate Gaussian distributions \cite{Chatterjee-2008,Reinert-2009,Meckes-2009} (the thesis of D\"{o}bler offers an excellent overview of this \cite{Dobler-2012}), from which the following theorem is obtained.

\begin{theorem}\label{Thm: Stein theorem}Let $(M,M')$ be an exchangeable pair of $N \times N$ random self-adjoint matrices with $\alpha_N$ a constant depending only on $N$ and $Z$ the multi-dimensional Gaussian random variable in Theorem \ref{Thm: ITE theorem}. If the random variable $Y(M) = (Y_2(M),\ldots,Y_n(M))$ satisfies
\begin{eqnarray}
\frac{1}{\alpha_N}\E[\delta Y_n|M] & = & - n Y_n(M) + R_n(M) \label{Eqn: Drift term} \\
\frac{1}{\alpha_N}\E[\delta Y_n \delta Y_m |M] & = &  2n^2 \delta_{nm} + R_{nm}(M) \label{Eqn: Diffusion term} \\
\frac{1}{\alpha_N}\E[ |\delta Y_n \delta Y_m \delta Y_l | |M] & =& R_{nml}(M).  \label{Eqn: Remainder term}
\end{eqnarray}
Then for all $\phi \in C^2(\R^k)$ we have
\begin{equation}\label{Eqn: Quantitative bound}
|\E[\phi(Y(M))] - \E[\phi(Z)]| \leq c_1\Rc^{(1)}\| \phi \| + c_2\Rc^{(2)} \|\nabla \phi \| + c_3\Rc^{(3)} \|\nabla^2\phi\|,
\end{equation}
where $\|\nabla^j \phi \|$ is given in (\ref{Eqn: Sup deriv norm}), $c_j$ are fixed positive constants and $\Rc^{(j)} = \sum_{n_1,\ldots,n_j}\E|R_{n_1\ldots n_j}(M)|$.
\end{theorem}

\begin{proof} Theorem \ref{Thm: Stein theorem} is a specific form of Theorem 3 in \cite{Meckes-2009}, except that we have decided to use the alternative quantities $\|\nabla^k \phi \|$ as bounds since these are easier to state and make minimal difference in the outcome of our main results. We have therefore decided to include the proof of Theorem \ref{Thm: Stein theorem} in Appendix \ref{Sec: Stein solution} for completeness and to aid the understanding of the interested reader, even though, beyond minimal adjustments, there is nothing new.
\end{proof}

\begin{remark}As was first noted by G\"{o}tze \cite{Gotze-1991} and Barbour \cite{Barbour-1990}, the operator $\A$ is the generator for a specific multi-dimensional Ornstein-Uhlenbeck (OU) process. Thus, in essence, Theorem \ref{Thm: Stein theorem} is stating that if the random walk is close (i.e. the remainders $R_n, R_{nm}$ etc and the constant $\alpha_N$ go to 0 in the limit of large $N$) to that of the associated OU-process then the corresponding stationary distributions will also be close - in the distributional sense of (\ref{Eqn: Quantitative bound}). This association is described in more detail in Section 4 of \cite{Joyner-2017}.
\end{remark}

\begin{remark}In principle one could remove the factor of $n$ present in (\ref{Eqn: FP eqn}) and achieve the same stationary distribution but it will transpire the evolution of our observables $Y_n(B)$, given in (\ref{Eqn: Centred Chebyshev}), can only be analysed if it is included. This is because this factor corresponds to rescaling the time $t \to nt$, which is independent of the random variable in question. Thus, in general, the linear statistic $\Phi_h(H)$ will not evolve according to a single one-dimensional OU process, but rather a linear combination of independent one-dimensional OU processes evolving at different rates.
\end{remark}

The novel aspect of our work concerns the evaluation of the remainders $R_n(H)$, $R_{nm}(H)$ and $R_{nml}(H)$. For comparison, the CLT results in \cite{Dobler-2012,Webb-2016,Lambert-2017}, whilst slight stronger, heavily utilise Dyson Brownian motion, which affords a closed form expression for the evolution of spectrum. In other words, the remainders are functions of the eigenvalues, i.e. $R_n(H) = R_n(\lambda_1(H),\ldots,\lambda_N(H))$ etc. However, since our ensembles are not invariant under, say unitary or orthogonal transformations, we do not have this luxury. We therefore use alternative combinatorial methods to obtain estimates in terms of the matrix dimension $N$.

The starting point of these methods comes from a generalised form of the Bartholdi identity, developed in \cite{Oren-2009} to obtain a trace formula for the eigenvalues of (magnetic) regular graphs. This allows us to relate the centred Chebyshev Polynomials $Y_n(H)$ to sums of products of matrix elements, like $H_{p_1p_2}H_{p_2p_3}\ldots$, associated to non-backtracking cycles (see Section \ref{Sec: Graph theoretical tools}). The change of such products under the appropriate random walks leads to remainder terms comprised of, again, certain classes of matrix products. Estimating the remainders consists of bounding the expectations of this quantities with respect to either the ITE and RITE. Here is where the combinatorial aspects arise, since, just as was first used by Wigner for showing convergence to the semi-circle distribution \cite{Wigner-1955,Wigner-1958}, one must evaluate the contributions arising from certain walks.

For the ITE the estimates are relatively straightforward because the matrix elements are independent. It means the contributions from many cycles are precisely zero. Those cycles that remain only give contributions tending to 0 in the large $N$ limit. For the RITE, however, a more complicated random walk leads, inevitably, to more complicated expressions for the remainder terms. Moreover, the lack of independence means the expectations of matrix products that were identically zero for the ITE are no longer so for the RITE. A key part of our analysis is therefore showing the correlations are small enough so the expectations go to zero sufficiently fast in $N$ (see Lemma \ref{Lem: Regular tournament expectation}). This is achieved by adapting McKay's methods \cite{McKay-1990} for the number of regular tournaments. Specifically, we transform the expectation of matrix products into a multi-dimensional integral, which are shown to be of a certain order in $1/N$.

\section{Graph theoretical tools}\label{Sec: Graph theoretical tools}

Before proceeding to our random walks we first introduce some necessary terminology and simple results. A graph $G$ consists of a set of vertices $V(G)$ and edges $E(G)$ connecting these vertices. $G$ is said to be \emph{simple} if every pair of vertices is connected by at most one edge and there are no vertices connected to themselves. $G$ is also said to be \emph{complete} if every pair of vertices has precisely one edge connecting them.

A \emph{walk} $\omega$ of length $n$ on a graph $G$ is an ordered sequence of vertices $\omega = (p_0,p_1,\ldots,p_{n-1},p_n)$ such that $p_{i+1} \neq p_i$ and all pairs $(p_i,p_{i+1}) \in E(G), i=0,\ldots,n-1$ are edges on the graph. If $p_{i+2} = p_i$ for some $i=0,\ldots,n-2$ then the walk is said to be \emph{backtracking}. Otherwise $\omega$ is \emph{non-backtracking}. A walk is also a \emph{cycle} (of length $n$) if the first and last vertices are the same, i.e. $p_0 = p_n$. Note that, in the present article, cycles will be distinguished by the starting vertex, so for example, $\omega = (1,2,3,4,1) \neq (2,3,4,1,2) = \omega'$. Again, the cycle is \emph{backtracking} if there exists some $i$ such that $p_i = p_{i+2(n)}$ and \emph{non-backtracking} otherwise.

We use the notations $V_{\omega}$ and $E_{\omega}$ to denote the set of distinct vertices and edges in a walk $\omega$ and $\nu_{\omega}(e)$ for the number of times the edge $e = (p,q)$ appears in $\omega$. Therefore, in terms of the tournament matrix $H$, a walk $\omega$ corresponds to the product over all matrix elements associated to (directed) edges in $\omega$, i.e
\begin{equation}\label{Eqn: Matrix walk definition}
H_{\omega} := H_{p_0p_1}H_{p_1p_2} \ldots H_{p_{n-1}p_n}.
\end{equation}
In addition, for a collection of walks $\omega_1,\omega_2,\ldots, \omega_n$ on $G$ we define
\begin{equation}\label{Eqn: Multiple walk vertices}
V_{\omega_1,\ldots,\omega_n} : = V_{\omega_1} \cup V_{\omega_2} \cup \ldots \cup V_{\omega_n}, \qquad E_{\omega_1,\ldots,\omega_n} : = E_{\omega_1} \cup E_{\omega_2} \cup \ldots \cup E_{\omega_n}
\end{equation}
and $\nu_{\omega_1,\ldots,\omega_n}(e)$ for the number of times the edge $e$ is traversed by the walks $\omega_1,\ldots,\omega_n$. Similarly
\begin{equation}\label{Eqn: Multiple walk matrices}
H_{\omega_1,\ldots,\omega_n} : = H_{\omega_1} H_{\omega_2}  \ldots H_{\omega_n}.
\end{equation}
If an edge $(p,q)$ appears an even number of times in $\omega_1,\ldots,\omega_n$ then it will be removed from $(\ref{Eqn: Multiple walk matrices})$ since we have identically $H_{pq}^2 = -H_{pq}H_{qp} = 1$ for every $H \in \TT_N$. It is therefore convenient to define the set of `free' edges as 
\[
F_{\omega_1,\ldots,\omega_n} := \{e \in E_{\omega_1,\ldots,\omega_n} : \nu_{\omega_1,\ldots,\omega_n}(e) \equiv 1 \mod(2)\},
\]
i.e. the set of edges that are traversed an odd number of times by $\omega_1,\ldots,\omega_n$. This will be especially useful when evaluating remainders for the RITE in Section \ref{Sec: RITE}.

We say that two walks $\omega$ and $\omega'$ are equivalent if $\omega'$ can be obtained from $\omega$ by simply relabelling the vertices and we will use the notation $\omega \sim \omega'$ to denote that is the case. For example $\omega = (1,2,3,1,4) \sim (2,3,9,2,6) = \omega'$. We will write $[\omega] := \{\omega' : \omega \sim \omega'\}$ to denote the associated equivalence class and if $\Omega$ is a set of walks then $[\Omega] := \{ [\omega] : \omega \in \Omega\}$ is the set of equivalence classes. Moreover, we shall use the notation $\omega \cong \omega'$ if $\omega \sim \omega'$ and $F_{\omega} = F_{\omega'}$. For example $\omega = (1,2,3,4,5,6,7,5,4,3,1) \cong (1,3,2,8,6,5,7,6,8,2,1) = \omega'$. The above notions immediately generalise to collections of walks $(\omega_1,\ldots,\omega_n)$.

\begin{lemma}\label{Eqn: Subgraph Betti relation}Let $G$ be a simple, connected graph with vertex set $V(G)$ and edge set $E(G)$. Let $G' \subseteq G$ be a subgraph with $V(G') \subseteq V(G)$ and $E(G') \subseteq E(G)$. Then 
\begin{equation}\label{Eqn: E-V inequality}
|E(G')| - |V(G')| \leq |E(G)| - |V(G)|,
\end{equation}
provided $|V(G')| \geq 1$.
\end{lemma}
\begin{proof}Let $C$ denote the number of connected components of $G'$. We can create a new graph $\tilde{G} \subseteq G$ by adding a minimal number of edges to $G'$ such that $\tilde{G}$ is connected, then
\[
|E(G')| - |V(G')| = |E(\tilde{G})| - |V(\tilde{G})| - C + 1 = \beta(\tilde{G}) - C.
\]
Here $\beta(\tilde{G}) = |E(\tilde{G})| - |V(\tilde{G})| + 1$ is the first Betti number of $\tilde{G}$, which counts the number of fundamental cycles. However, since $\tilde{G}$ is a subgraph of $G$ it cannot have more fundamental cycles than $G$ and so
\[
|E(G')| - |V(G')| \leq |E(G)| - |V(G)| + (1-C).
\]
The condition $|V(G')| \geq 1$ ensures that $C \geq 1$, which completes the result.
\end{proof}

\begin{corollary}\label{Cor: RITE vertex edge identity}Let $\bar{\omega} = (\omega_1,\ldots,\omega_n)$ be a collection of walks and define the subgraph $G = (V_{\bar{\omega}},F_{\bar{\omega}})$. If $G$ is disconnected with $C$ components then we write $G_i = (V^{(i)}_{\bar{\omega}},F^{(i)}_{\bar{\omega}}), i=1,\ldots,C$ to denote the subgraphs of these components and $\beta_i = |F^{(i)}_{\bar{\omega}}| -  |V^{(i)}_{\bar{\omega}}| + 1$ the associated first Betti numbers. Suppose $\bar{\omega} \sim \bar{\omega}'$ then
\[
|V_{\bar{\omega},\bar{\omega}'}| - \frac{|F_{\bar{\omega},\bar{\omega}'}|}{2} \leq |V_{\bar{\omega}}| + \sum_{i=1}^C \delta_{\beta_i,0}.
\]
\end{corollary}
\begin{proof}By construction we have 
\begin{align}
|V_{\bar{\omega},\bar{\omega}'}| & = |V_{\bar{\omega}}| + |V_{\bar{\omega}'}| - |V_{\bar{\omega}}\cap V_{\bar{\omega}'}| = 2 |V_{\bar{\omega}}| - |V_{\bar{\omega}}\cap V_{\bar{\omega}'}| \\
|F_{\bar{\omega},\bar{\omega}'}| & = |F_{\bar{\omega}}| + |F_{\bar{\omega}'}| - 2|F_{\bar{\omega}}\cap F_{\bar{\omega}'}| = 2 |F_{\bar{\omega}}| - 2|F_{\bar{\omega}}\cap F_{\bar{\omega}'}|
\end{align}
and so
\[
|V_{\bar{\omega},\bar{\omega}'}| - \frac{|F_{\bar{\omega},\bar{\omega}'}|}{2} = 2 |V_{\bar{\omega}}| - |F_{\bar{\omega}}| + |F_{\bar{\omega}}\cap F_{\bar{\omega}'}| - |V_{\bar{\omega}}\cap V_{\bar{\omega}'}|.
\]
Now, the graph $G' = (V_{\bar{\omega}}\cap V_{\bar{\omega}'}, F_{\bar{\omega}}\cap F_{\bar{\omega}'}) \subseteq G$. Let us suppose $G$ is connected (i.e. $C=1$) and $|V_{\bar{\omega}}\cap V_{\bar{\omega}'}| \geq 1$, then by Lemma \ref{Eqn: Subgraph Betti relation} we have 
\[
|V_{\bar{\omega},\bar{\omega}'}| - \frac{|F_{\bar{\omega},\bar{\omega}'}|}{2} \leq |V_{\bar{\omega}}|.
\]
It thus remains to check the case when $|V_{\bar{\omega}}\cap V_{\bar{\omega}'}| =0$. In this case $|F_{\bar{\omega}}\cap F_{\bar{\omega}'}|=0$ and so
\[
|V_{\bar{\omega},\bar{\omega}'}| - \frac{|F_{\bar{\omega},\bar{\omega}'}|}{2} = 2|V_{\bar{\omega}}| - |F_{\bar{\omega}}| = |V_{\bar{\omega}}| - 1 + \beta(G) \leq |V_{\bar{\omega}}| + \delta_{\beta,0}(G).
\]
Extending this to $C$ connected components completes the result.
\end{proof}

For our imaginary tournament matrices there is an intimate connection between the traces of Chebyshev polynomials (see Equation (\ref{Eqn: Chebyshev traces})) and the sets of non-backtracking cycles. This is given by the following lemma.

\begin{lemma}\label{Lem: Non-backtracking}Let $M$ be an $N \times N$ self-adjoint matrix with elements of the form
\begin{equation}
M_{pq} = e^{i\phi_{pq}} = \overline{M}_{qp}, \hspace{10pt} \phi_{pq} \in [0,2\pi), \forall q \neq p
\end{equation}
and $M_{pp} = 0$ for all $p$. Then

\begin{equation}\label{Eqn: Chebyshev trace}
\Tr\left[T_n\left(\frac{M}{2\sqrt{N-2}}\right)\right] = \frac{1}{2} \frac{1}{(N-2)^{\frac{n}{2}}} \left[\sum_{\omega \in \Omega_n} M_\omega  - \frac{1}{2}(N-3)(1 + (-1)^n)\right],
\end{equation}
where $\Omega_{n}$ denotes the set of non-backtracking cycles of length $2n$ and $M_{\omega}$ is given in (\ref{Eqn: Matrix walk definition}).
\end{lemma}

\begin{proof}
We are aware of two related methods for proving the validity of this statement that we shall not recount here. The first approach is to make a generalisation of the so-called Bartholdi identity (see e.g. \cite{Oren-2009,Joyner-2017}) that relates the spectrum of $M$ to another matrix associated to non-backtracking walks in the edge space. This connection is applicable since $M$ can be considered as a magnetic adjacency matrix of a complete graph on $N$ vertices. The second approach is based upon showing that polynomials associated to non-backtracking walks obey the same recursion relations as the Chebyshev polynomials (see e.g. \cite{Sodin-2017} and references therein).
\end{proof}

\section{Imaginary tournament ensemble}\label{Sec: ITE}

We now construct the random walk process in $\TT_N$. Many of the intricate details of this walk are discussed in \cite{Joyner-2015} and so we attempt to keep to the essential points. Suppose that at time $t \in \N$ we select a matrix $H \in \TT_N$, then at time $t+1$ we randomly choose another matrix $H' \in \TT_N$ by selecting with equal probability one of the upper triangular elements of $H$ (say $H_{pq}$ with $p <q$) and, together with its symmetric partner (i.e. $H_{qp}$), we change its sign $H_{pq} \to -H_{pq}$. We will write
\begin{equation}\label{Eqn: Matrix change ITE}
\delta H^{pq} := H' - H = -2H_{pq}[\e_p\e_q^T - \e_q\e_p^T],
\end{equation}
to denote the $N \times N$ rank 2 difference matrix obtained as a result of performing this change of sign. Here $\e_p$ is the column vector with a 1 in entry $p$ and 0 everywhere else. This switch corresponds to changing the direction of an edge (see Figure \ref{Fig: RW for ITE}) in the associated tournament graph, as described in Section \ref{Sec: Definitions and results}. 

\begin{figure}[ht]
\centerline{(a)\includegraphics[width=0.28\textwidth]{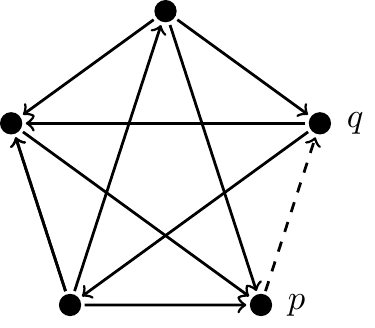}
\hspace{70pt}
(b)\includegraphics[width=0.28\textwidth]{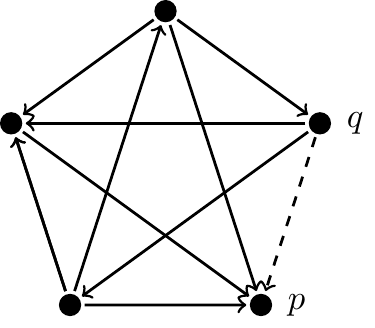}}
\caption{The Markov process consists of choosing an edge $(p,q)$ uniformly at random in the tournament graph (a) and then switching the orientation to obtain the tournament graph in (b). In this example the $(p,q)$-th element of the associated adjacency matrix $A_{pq} = 1 - A_{qp} = 1$ in (a) is updated to $A_{pq} = 0$ in (b). Hence $H_{pq} = i(2A - (\mathbf{E}_N-\mathbf{I}_N))_{pq} = -H_{qp} = i \mapsto H_{pq} = -i$ when making the switch from (a) to (b).}
\label{Fig: RW for ITE}
\end{figure}

Interpreting this in terms of a random walk we say that if the walker is at $H$ at time $t$ then in each unit time step we let the walker move to any matrix $H' \in \TT_N$ which is exactly a Hamming distance\footnote{The Hamming distance between two matrices $H,H' \in \TT_N$ is given by $|H - H'| = \frac{1}{2}\sum_{p<q}|H'_{pq} - H_{pq}|$, which counts the number of differences in signs of the free matrix elements.} one away with equal probability - giving us the transition probability
\begin{equation}\label{Eqn: RTE motion}
\rho(H \to H') = \left\{\begin{array}{ll}
\frac{1}{d_N} & |H -H'| = 1 \\
0 & |H -H'| \neq 1, \end{array} \right.
\end{equation}
where $d_N = N(N-1)/2$ is the number of independent elements of $H$. Therefore, if $P_t(H)$ is the probability for the random walker to be at matrix $H$ at time $t$ then the probability to be at some other matrix $H' \in \TT_N$ is given by
\[
P_{t+1}(H') = \sum_{H \in \TT_N} \rho(H \to H')P_t(H) = \sum_{H : |H' - H| = 1} \frac{P_t(H)}{d_N} = \frac{1}{d_N} \sum_{p < q} P_t(H' - \delta H^{(pq)}).
\]
One may then verify easily that $P_t(H') = |\TT_N|^{-1}$ (the measure of the ITE in Definition \ref{Dfn: ITE}) is the stationary distribution of this process, since $\#\{H : |H' - H| = 1\} = d_N$. In this instance the random matrices $H$ and $H'$ have the same distribution and are thus an exchangeable pair.

The expected change of some observable $f(H)$ with respect to this random walk is hence given by
\begin{equation}\label{Eqn: Observable change ITE}
\E[\delta f|H] := \sum_{H' \in \TT_N} \rho(H \to H')[f(H') - f(H)] = \frac{1}{d_N}\sum_{p<q}[f(H+\delta H^{pq}) - f(H)].
\end{equation}
Similarly, higher moments are obtained by taking the expectation of products of changes, i.e. for $f_1(H),f_2(H),\ldots,f_k(H)$
\begin{equation}\label{Eqn: Multiple observables change ITE}
\E[\delta f_1 \ldots \delta f_k | H] := \frac{1}{d_N}\sum_{p < q} [f_1(H + \delta H^{pq}) - f_1(H)] \ldots [f_k(H + \delta H^{pq}) - f_k(H)].
\end{equation}
We are now in position to state how the observables $Y_n(H)$, given in (\ref{Eqn: Centred Chebyshev}), behave under this random walk.

\begin{proposition}\label{Prop: ITE evolution}Let $(H',H)$ be an exchangeable pair from the ITE with distribution $P(H) = |\TT_N|^{-1}$ and connected via (\ref{Eqn: RTE motion}). Let $Y_n(H)$ be as defined in (\ref{Eqn: Centred Chebyshev}). Then
\begin{enumerate}
\item \label{Eqn: ITE drift} $\frac{d_N}{4}\E[\delta Y_n|H] = -nY_n(H) + R_n(H)$ (Drift term)
\item \label{Eqn: ITE diffusion}$\frac{d_N}{4}\E[\delta Y_n\delta Y_m|H] = 2n^2\delta_{nm} + R_{nm}(H)$ (Diffusion term)
\item \label{Eqn: ITE remainder}$\frac{d_N}{4}\E[|\delta Y_n\delta Y_m\delta Y_l | |H] = R_{nml}(H)$ (Remainder term)
\end{enumerate}
with $\E|R_n(H)| = \O(N^{-1})$, $\E|R_{nm}(H)| = \O(N^{-1})$ and $\E|R_{nml}(H)| = \O(N^{-1})$ for all $n,m,l = 2,\ldots,k$.
\end{proposition}

\begin{proof}
The proofs for Parts \ref{Eqn: ITE drift}, \ref{Eqn: ITE diffusion} and \ref{Eqn: ITE remainder} will be presented in Sections \ref{Sec: ITE Drift term}, \ref{Sec: ITE Diffusion term} and \ref{Sec: ITE Remainder term} respectively.
\end{proof}

To show Theorem \ref{Prop: ITE evolution} we utilise Lemma \ref{Lem: Non-backtracking}, which allows us to express $Y_n(H)$ in the following form
\begin{equation}\label{Eqn: Y formula ITE}
Y_n(H) = \Tr[T_{2n}(H)] - \E[\Tr[T_{2n}(H)]] = \frac{1}{2}\frac{1}{(N-2)^n}\sum_{\omega \in \Lambda_{2n}} H_{\omega}.
\end{equation}
Here $\Lambda_{2n} := \{\omega \in \Omega_{2n} : \exists e  \in E_{\omega} \ \text{s.t.} \ \nu_{\omega}(e) = 1 \mod (2) \}$ is the set of non-backtracking cycles $\omega \in \Omega_{2n}$ for which there is at least one edge that is traversed an odd number of times. Hence, since the matrix elements are independent, $\E[H_{\omega}] = 0$ for all $\omega \in \Lambda_{2n}$, which is why this term disappears in (\ref{Eqn: Y formula ITE}).

From (\ref{Eqn: Matrix change ITE}) for some edge $e = (p',q')$
\begin{equation}\label{Eqn: ITE matrix element change}
\delta H^{pq}_e := \e_{p'}^T(\delta H^{pq})\e_{q'} = -2H_{pq}(\delta_{pp'}\delta_{qq'} + \delta_{qp'}\delta_{pq'}) = -2H_e \chi_{e,pq}.
\end{equation}
Here $\chi_{e,pq}$ is the indicator, equal to 1 if $e = (p,q)$ or $(q,p)$ and 0 otherwise. Therefore, if $\omega = (p_0,\ldots,p_{2n-1},p_0)$ then
\[
\delta H^{pq}_\omega := \prod_{i = 0}^{2n-1} (H + \delta H^{pq})_{p_i p_{i+1(2n)}} - H_{\omega} = H_{\omega} \prod_{i = 0}^{2n-1} (1 - 2 \chi_{p_i p_{i+1(2n)},pq}) - H_{\omega} = -2H_\omega \phi_{\omega,pq}
\]
where
\[
\phi_{\omega,pq} = \frac{1}{2}((-1)^{\chi_{\omega,pq}} - 1), \qquad \chi_{\omega,pq} = \sum_{i=0}^{2n-1} \chi_{p_i p_{i+1(2n)},pq}.
\]
Hence $\phi_{\omega,pq} = 1$ if $\nu_{\omega}((p,q)) = 1 \mod (2)$ and 0 otherwise. In other words $\phi_{\omega,pq}$ is only non-zero when the cycle $\omega$ traverses the undirected edge $(p,q)$ an odd number of times.

Therefore $\delta Y^{pq}_n := Y_n(H + \delta H^{pq}) - Y_n(H)$ is given by
\begin{equation}\label{Eqn: delta Y ITE expression}
\delta Y^{pq}_n   = \frac{1}{2}\frac{1}{(N-2)^n}\sum_{\omega \in \Lambda_{2n}}\delta H^{pq}_\omega 
 = -\frac{1}{(N-2)^n}\sum_{\omega \in \Lambda_{2n}} H_{\omega}\phi_{\omega,pq}.
\end{equation}

\subsection{Proof of Proposition \ref{Prop: ITE evolution} Part \ref{Eqn: ITE drift} - Drift term}\label{Sec: ITE Drift term}

Inserting the form (\ref{Eqn: delta Y ITE expression}) for $\delta Y^{pq}_n$ into the expression (\ref{Eqn: Observable change ITE}) for the expected change of an observable undergoing this random walk leads to
\begin{equation}
\E[\delta Y_n |H]  = \frac{1}{d_N} \sum_{p < q} \delta Y^{pq} 
 = - \frac{1}{d_N}\frac{1}{(N-2)^n} \sum_{p < q} \sum_{\omega \in \Lambda_{2n}} H_{\omega} \phi_{\omega,pq} .
 = \frac{4}{d_N} [-nY_n(H) + R_n(H)],
\end{equation}
Using the expression (\ref{Eqn: Y formula ITE}) for $Y_n(H)$ therefore gives the remainder
\begin{equation}\label{Eqn: Drift remainder ITE expression 1}
R_n(H) = \frac{n}{2}\frac{1}{(N-2)^n}  \sum_{\omega \in \Lambda_{2n}} H_{\omega}\bigg(1 - \frac{1}{2n}\sum_{p < q} \phi_{\omega,pq}\bigg) .
\end{equation}
Our aim is to show that $|\E[R_n(H)]| = \O(N^{-1})$. We now write $\Lambda^{\star}_{2n} = \{\omega \in \Lambda_{2n} : |F_{\omega}| = 2n\}$, i.e. the set of non-backtracking cycles in $\Lambda_{2n}$ in which all edges are traversed exactly once. We also write $\Lambda^{\circ}_{2n} = \Lambda_{2n} \setminus \Lambda^{\star}_{2n}$ for the set of non-backtracking cycles in which at least one edge is traversed more than once. Importantly, for all $\omega \in \Lambda^{\star}_{2n}$ we have
\begin{equation}\label{Eqn: Free edge ITE identity}
\sum_{p < q} \phi_{\omega,pq} = 2n.
\end{equation}
Therefore the sum over $\omega$ in $\Lambda_{2n}$ in (\ref{Eqn: Drift remainder ITE expression 1}) can be reduced to the lesser sum over $\Lambda^{\circ}_{2n}$. As outlined in Section \ref{Sec: Graph theoretical tools}, let us write $[\omega]$ for the equivalence class of vertex labelings of non-backtracking cycles and $[\Lambda^{\circ}_{2n}]$ for the set of such equivalence classes in $\Lambda^{\circ}_{2n}$. Given that $\sum_{p < q} \phi_{\omega,pq}$ is the same for all $\omega \in [\omega]$
\begin{multline}
\E|R_n(H)| \leq  \O(N^{-n})\E\bigg|\sum_{[\omega] \in [\Lambda^{\circ}_{2n}]} \bigg(1 - \frac{1}{2n}\sum_{p < q} \phi_{\omega,pq}\bigg)\sum_{\omega \in [\omega]} H_{\omega} \bigg| 
\\
\leq \O(N^{-n})\sum_{[\omega] \in [\Lambda^{\circ}_{2n}]} \bigg|1 - \frac{1}{2n}\sum_{p < q} \phi_{[\omega],pq}\bigg|\E\bigg|\sum_{\omega \in [\omega]}H_{\omega}\bigg|.
\end{multline}
$\frac{1}{2n}\sum_{p < q}\phi_{[\omega],pq} < 2n = \O(1)$ for all $[\omega] \in  [\Lambda^{\circ}_{2n}]$ so, using the inequality $\E|A| \leq \sqrt{\E[A^2]}$,
\[
\E|R_n(H)| \leq \O(N^{-n})\sum_{[\omega] \in [\Lambda^{\circ}_{2n}]} \sqrt{ \sum_{\omega,\omega' \in [\omega]} \E[H_{\omega,\omega'}]}.
\]
where $H_{\omega,\omega'} := H_{\omega}H_{\omega'}$, as in (\ref{Eqn: Multiple walk matrices}). Since $\omega \in \Lambda^{\circ}_{2n}$ there must be a least one edge that is traversed twice (i.e. $|F_{\omega}| \leq 2n-2$) - reducing the number of vertices in $V_{\omega}$ such that $|V_{\omega}| \leq 2n-2$. Therefore, because the quantity $\E[H_{\omega,\omega'}] \neq 0$ only when $\omega \cong \omega'$ (meaning $V_{\omega} = V_{\omega'}$), those contributing pairs $(\omega,\omega')$ satisfy $|V_{\omega,\omega'}| \leq 2n-2$. The contribution from the term inside the square-root is thus obtained by labelling the independent vertices in $V_{\omega,\omega'}$, exactly as done by Wigner \cite{Wigner-1955}. Up to a constant, we have $N(N-1)\ldots (N- |V_{\omega,\omega'}| -1) = \O(N^{|V_{\omega,\omega'}|})$ pairs $(\omega,\omega') \in [\omega]$ such that $\E[H_{\omega,\omega'}] = H_{\omega,\omega'} = 1$, so taking the square root we have $\E|R_n(H)| \leq \O(N^{-n}) | [\Lambda^{\circ}_{2n}]| \sqrt{\O(N^{2n-2})} = | [\Lambda^{\circ}_{2n}]| \O(N^{-1})$. We are thus left to evaluate $| [\Lambda^{\circ}_{2n}]|$, the number of \emph{unlabelled} non-backtracking cycles $\omega \in \Lambda^{\circ}_{2n}$. However, since the labelling has been removed this quantity is now independent of $N$, and so $| [\Lambda^{\circ}_{2n}]| = \O(1)$, meaning $\E|R_n(H)| = \O(N^{-1})$, as desired.

\subsection{Proof of Proposition \ref{Prop: ITE evolution} Part \ref{Eqn: ITE diffusion} - Diffusion term}\label{Sec: ITE Diffusion term}

Similar to the proof of the drift term, we start by inserting the form (\ref{Eqn: delta Y ITE expression}) for $\delta Y^{pq}_n$ into the expression (\ref{Eqn: Multiple observables change ITE}) for the expected change of multiple observables, leading to the following diffusion term
\[
\E[\delta Y_n \delta Y_m |H] = \frac{1}{d_N}\sum_{p < q} \delta Y^{pq}_n \delta Y^{pq}_m = \frac{1}{d_N}  \frac{1}{(N-2)^{n+m}} \sum_{p < q} \sum_{\omega_1 \in \Lambda_{2n}} \sum_{\omega_2 \in \Lambda_{2m}} H_{\omega_1,\omega_2}\phi_{\omega_1,pq}\phi_{\omega_2,pq}.
\]
Therefore, if $\E[\delta Y_n \delta Y_m |H] = \frac{4}{d_N}[2n^2\delta_{nm} + R_{nm}(H)]$, then
\begin{equation}\label{Eqn: ITE Diffusion remainder expression}
R_{nm}(H) = \frac{1}{2}\frac{1}{(N-2)^{n+m}} \sum_{p < q} \sum_{\omega_1 \in \Lambda_{2n}} \sum_{\omega_2 \in \Lambda_{2m}} H_{\omega_1,\omega_2}\phi_{\omega_1,pq}\phi_{\omega_2,pq} - 2n^2\delta_{nm}.
\end{equation}
We estimate the cases $n=m$ and $n \neq m$ separately. For the former case let us take $\Lambda^{\star}_{2n}$ as in Section \ref{Sec: ITE Drift term} and define $\Gamma^{\star}_{2n} = \{ (\omega_1,\omega_2) \in \Lambda^{\star}_{2n} \times \Lambda^{\star}_{2n} : \omega_1 \cong \omega_2 \}$, with the complement $\Gamma^{\circ}_{2n} = (\Lambda_{2n} \times \Lambda_{2n}) \setminus \Gamma^{\star}_{2n}$. For walks $\omega_1 \cong \omega_2$ we have $\phi_{\omega_1,pq} = \phi_{\omega_2,pq}$ and $H_{\omega_1} = H_{\omega_2}$ (so $H_{\omega_1,\omega_2} = 1$). Moreover, if $\omega_1 \in \Lambda^{\star}_{2n}$ then from (\ref{Eqn: Free edge ITE identity}) $\sum_{p<q} \phi_{\omega_1,pq}^2 = \sum_{p<q} \phi_{\omega_1,pq} = 2n$. In addition, if $|V_{\omega_1}| = |V_{\omega_1}| = 2n$ ($\omega$ is a single loop) then for a fixed $\omega_1$ there are $4n$ possible $\omega_2$ such that $\omega_1 \cong \omega_2$ - obtained by choosing the $2n$ possible starting vertices of the cycle and the 2 possible orientations. Labelling the independent vertices of $\omega_1$ leads to a contribution to $|\Gamma^{\star}_{2n}|$ of $4nN(N-1)\ldots (N- (2n-1)) = 4nN^{2n} + \O(N^{2n-1})$. If $|V_{\omega_1}| = |V_{\omega_1}| < 2n$ then the contribution to $|\Gamma^{\star}_{2n}|$ will be of order $\O(N^{2n-1})$. Therefore
\begin{multline}
R_{nn}(H) = \frac{1}{4}\frac{1}{(N-2)^{2n}} \sum_{(\omega_1,\omega_2) \in \Gamma^{\circ}_{2n}} H_{\omega_1,\omega_2} \alpha_{\omega_1,\omega_2} + \frac{n}{2(N-2)^{2n}}|\Gamma^{\star}_{2n}|  - 2n^2 \\
= \frac{1}{4}\frac{1}{(N-2)^{2n}} \sum_{(\omega_1,\omega_2) \in\Gamma^{\circ}_{2n}} H_{\omega_1,\omega_2}\alpha_{\omega_1,\omega_2}  + \O(N^{-1}),
\end{multline}
where $\alpha_{\omega_1,\omega_2} : = \sum_{p < q} \phi_{\omega_1,pq}\phi_{\omega_2,pq}$. Using that $\alpha_{\omega_1,\omega_2}$ is the same for all $\omega_1,\omega_2 \in [\omega_1,\omega_2]$, we find
\begin{multline}\label{Eqn: ITE Rnn expression}
\E|R_{nn}(H)| \leq \O(N^{-2n}) \sum_{[\omega_1,\omega_2] \in [\Gamma^{\circ}_{2n}]} \alpha_{[\omega_1,\omega_2]} \E\bigg| \sum_{(\omega_1,\omega_2) \in [\omega_1,\omega_2]} H_{\omega_1,\omega_2}\bigg| + \O(N^{-1}) \\
\leq \O(N^{-2n}) \sum_{[\omega_1,\omega_2] \in [\Gamma^{\circ}_{2n}]} \alpha_{[\omega_1,\omega_2]} \sqrt{\sum_{(\omega_1,\omega_2),(\omega'_1,\omega'_2) \atop \in [\omega_1,\omega_2]} \E[H_{\omega_1,\omega_2,\omega'_1,\omega'_2}]} + \O(N^{-1}).
\end{multline}
Since we want to maximise the number of vertices the main contribution to the above will come from cycles $\omega_1$, $\omega_2 \in \Lambda^{\star}_{2n}$ in which $|V_{\omega_1}| = |V_{\omega_2}| = 2n$ (i.e. all vertices are distinct). However, $\omega_1$ and $\omega_2$ must share at least one edge (otherwise $\alpha_{\omega_1,\omega_2} = 0$) and we cannot have $V_{\omega_1} = V_{\omega_2}$, otherwise $\omega_1 \cong \omega_2$ meaning $(\omega_1,\omega_2) \notin \Gamma^{\circ}_{2n}$. The subgraph $\hat{G} = (V_{\omega_1,\omega_2},E_{\omega_1,\omega_2})$ is therefore connected and will contain edges that are traversed only once by $(\omega_1,\omega_2)$, implying that $\beta(\hat{G}) = 2$. Furthermore we have $2|E_{\omega_1,\omega_2}| - |F_{\omega_1,\omega_2}| = 4n$, so $|V_{\omega_1,\omega_2}| = |E_{\omega_1,\omega_2}| - \beta(\hat{G}) + 1 = 2n - 1 + |F_{\omega_1,\omega_2}|/2$. In particular $|V_{\omega_1,\omega_2}| - |F_{\omega_1,\omega_2}|$ is the number of isolated vertices in the subgraph $G = (V_{\omega_1,\omega_2},F_{\omega_1,\omega_2})$. The remaining vertices form a single loop connected by the edges $F_{\omega_1,\omega_2}$.

The quantity $\E[H_{\omega_1,\omega_2,\omega'_1,\omega'_2}]$ is only non-zero if $(\omega_1,\omega_2) \cong (\omega'_1,\omega'_2)$, i.e. $F_{\omega_1,\omega_2} = F_{\omega'_1,\omega'_2}$, therefore $|V_{\omega_1,\omega_2,\omega_1',\omega_2'}| \leq |F_{\omega_1,\omega_2}| + 2(|V_{\omega_1,\omega_2}| - |F_{\omega_1,\omega_2}|) =  4n - 2$. Again, we have $|[\Gamma^{\circ}_{2n}] | = \O(1)$, since it is independent of $N$ and also $\alpha_{\omega_1,\omega_2} = \O(1)$, since it is equal to, at most, the number of shared edges of $\omega_1$ and $\omega_2$. Hence,
$\E|R_{nn}(H)| \leq \O(N^{-2n})\sqrt{\O(N^{4n-2})} + \O(N^{-1}) = \O(N^{-1})$.

\vspace{10pt}

It thus remains to evaluate $\E|R_{nm}(H)|$ for $n \neq m$. In this instance we have, from (\ref{Eqn: ITE Diffusion remainder expression})
\begin{multline}
\E|R_{nm}(H)| \leq \O(N^{-(n+m)}) \E\bigg| \sum_{(\omega_1,\omega_2) \atop \in \Lambda_{2n} \times \Lambda_{2m}} H_{\omega_1,\omega_2} \alpha_{\omega_1,\omega_2} \bigg| \\
\leq  \O(N^{-(n+m)})\sum_{[\omega_1,\omega_2] \atop \in [\Lambda_{2n} \times \Lambda_{2m}]}\alpha_{[\omega_1,\omega_2]} \sqrt{ \sum_{(\omega_1,\omega_2),(\omega'_1,\omega'_2) \atop \in [\omega_1,\omega_2]} \E[H_{\omega_1,\omega_2,\omega'_1,\omega'_2}]}.
\end{multline}
Again, the main contribution will come from cycles $\omega_1$ and $\omega_2$ in which all vertices are distinct, i.e. $|V_{\omega_1}| = 2n$ and $|V_{\omega_2}| = 2m$.  However, since $n \neq m$, $\omega_1$ and $\omega_2$ cannot share all the same edges. The condition $\alpha_{\omega_1,\omega_2} >0$ only if $\omega_1$ and $\omega_2$ share at least one edge, and therefore, for the same reasons as above, those contributing collections of cycles $(\omega_1,\omega_2,\omega'_1,\omega'_2)$ for which $\E[H_{\omega_1,\omega_2,\omega'_1,\omega'_2}]$ is non-zero satisfy $|V_{\omega_1,\omega_2,\omega'_1,\omega'_2}| \leq 2n+2m -2$. Hence, $\E|R_{nm}(H)| \leq \O(N^{-(n+m)})|[\Lambda_{2n} \times \Lambda_{2m}]|\sqrt{\O(N^{2n+2m-2})} = \O(N^{-1})$.

\subsection{Proof of Proposition \ref{Prop: ITE evolution} Part \ref{Eqn: ITE remainder} - Remainder term}\label{Sec: ITE Remainder term}

For the remainder term we again insert the expression (\ref{Eqn: delta Y ITE expression}) into (\ref{Eqn: Multiple observables change ITE}), which gives us
\begin{multline}
\E[|\delta Y_n \delta Y_m \delta Y_l|  |H] = \frac{1}{d_N}\sum_{p < q} |\delta Y^{pq}_n \delta Y^{pq}_m \delta Y^{pq}_l| \\
\leq  \frac{1}{d_N}\frac{1}{(N-2)^{n+m+l}}\sum_{p < q} \bigg|\sum_{\omega_1,\omega_2, \omega_3 \atop \in \Lambda_{2n} \times  \Lambda_{2m} \times  \Lambda_{2l}} H_{\omega_1,\omega_2,\omega_3} \phi_{\omega_1,pq}\phi_{\omega_2,pq} \phi_{\omega_3,pq}\bigg|.
\end{multline}
Let us define $\Gamma^{pq}_{2n,2m,2l} := \{(\omega_1,\omega_2,\omega_3) \in \Lambda_{2n} \times  \Lambda_{2m} \times  \Lambda_{2l} : \phi_{\omega_1,pq}\phi_{\omega_2,pq} \phi_{\omega_3,pq} = 1\}$ as the set of non-backtracking cycles that all traverse the edge $(p,q)$ an odd number of times. Taking the expectation over the ITE subsequently leads to
\begin{multline}
\E|R_{nml}(H)| = \frac{d_N}{4}\E [\E[|\delta Y_n \delta Y_m \delta Y_l| |H]] 
= \O(N^{-(n+m+l)})\sum_{p<q}\E \bigg|\sum_{(\omega_1,\omega_2,\omega_3) \in \Gamma^{pq}_{2n,2m,2l}} H_{\omega_1,\omega_2,\omega_3} \bigg| \\
 \leq \O(N^{-(n+m+l)})\sum_{p<q} \sum_{ [\omega_1,\omega_2,\omega_3] \in [\Gamma^{pq}_{2n,2m,2l}]}    \sqrt{ \sum_{(\omega_1,\omega_2,\omega_3),(\omega'_1,\omega'_2,\omega'_3) \atop \in [\omega_1,\omega_2,\omega_3] }  \E[H_{\omega_1,\omega_2,\omega_3,\omega'_1,\omega'_2,\omega'_3}] }.
\end{multline}
The main contribution to the above will again come from non-backtracking cycles in which all vertices are distinct ($|V_{\omega_1}| = |V_{\omega'_1}| = 2n$ etc.), as this maximises the number of vertices. In this case all the cycles $\omega_i,\omega'_i, i=1,2,3$ must traverse the edge $p,q$ precisely once. The expectation $\E[H_{\omega_1,\omega_2,\omega_3,\omega'_1,\omega'_2,\omega'_3}]$ is only non-zero when every edge in $E_{\omega_1,\omega_2,\omega_3,\omega'_1,\omega'_2,\omega'_3}$ is traversed an even number of times by $(\omega_1,\omega_2,\omega_3,\omega'_1,\omega'_2,\omega'_3)$. Therefore the number of vertices will be maximised when every edge (other than $(p,q)$) is traversed precisely twice, in which case $|V_{\omega_1,\omega_2,\omega_3,\omega'_1,\omega'_2,\omega'_3}| = 2n+2m+2l - 4$. However the two vertices $p$ and $q$ are fixed, so when obtaining the contribution inside the square root above by labelling the vertices we get $\E|R_{nml}(H)| = \sum_{p < q} \O(N^{-(n+m+l)})\sqrt{\O(N^{2n+2m+2l - 4 - 2})} = \sum_{p < q}\O(N^{-3}) = \O(N^{-1})$.

\section{Regular imaginary tournament ensemble}\label{Sec: RITE}

In a similar manner to the previous section we shall introduce a random walk within $\RT_N$, which in turn induces a random walk in the variables $Y_n(H)$. Obviously this must be different to that of ITE in the previous section, for if we simply change the sign of one element of $H$ then we no longer have $\sum_{q}H_{pq} = 0$ for all $p$ and therefore the new matrix $H' \notin \RT_N$. To remedy this situation we use a random walk that has already been investigated previously in the literature \cite{Kannan-1997}. To describe this Markov process we first note that every regular tournament on $N$ vertices contains directed cycles $\bq = (q_0,q_1,q_2,q_0)$ of length 3, i.e. $H_{q_0q_1} = H_{q_1q_2} =H_{q_2q_0}$ (see e.g. Figure \ref{Fig: RW for RITE} (a)). We shall refer to such directed cycles as \emph{triangles}, for which there are precisely 
\begin{equation}\label{Eqn: No. directed triangles}
d_N = \frac{N(N-1)(N+1)}{4} 
\end{equation}
in every regular tournament. Note that we distinguish labelled triangles, so $(1,2,3,1) \neq (2,3,1,2)$ for example.

\begin{proof}[Proof of (\ref{Eqn: No. directed triangles})]Let us introduce the following indicator function
\begin{equation}\label{Eqn: Indicator function}
\Theta_{\bq}(H) = \frac{1}{8}(1 - H_{q_0q_1}H_{q_1q_2})(1 - H_{q_1q_2}H_{q_2q_0})(1 - H_{q_2q_0}H_{q_0q_1})(1- \delta_{q_0q_1}\delta_{q_1q_2}\delta_{q_2q_0}),
\end{equation}
which satisfies
\begin{equation}\label{Eqn: RT Indicator}
\Theta_{\bq}(H) = \begin{cases} 1 & H_{q_0q_1} = H_{q_1q_2} =H_{q_2q_0} \ \text{and} \ q_0 \neq q_1 \neq q_2 \neq q_0 \\
0 & \text{otherwise} \ .
\end{cases}
\end{equation}
Summing over $\bq$ and using that $H_{pq}H_{pq} = -1$ and $\sum_{r: r \neq p,q} H_{qr} = -H_{qp}$ gives  
\begin{align}\label{Eqn: Indicator expansion}
\sum_{\bq} \Theta_{\bq}(H) & =  \frac{1}{8}\sum_{q_0 \neq q_1 \neq q_2 \neq q_0}(1 - H_{q_0q_1}H_{q_1q_2})(1 - H_{q_1q_2}H_{q_2q_0})(1 - H_{q_2q_0}H_{q_0q_1}) \nonumber \\
& =  \frac{1}{8}\sum_{q_0 \neq q_1 \neq q_2 \neq q_0}(2  - 2H_{q_0q_1}H_{q_1q_2} - 2H_{q_1q_2}H_{q_2q_0} - 2H_{q_2q_0}H_{q_0q_1}) \nonumber \\
& = \frac{1}{4}\bigg[\sum_{q_0 \neq q_1 \neq q_2 \neq q_0}1 + 3\sum_{p \neq q} H_{qp}H_{pq}\bigg] = \frac{1}{4}\bigg[N(N-1)(N-2) + 3N(N-1)\bigg] = d_N.
\end{align}
\end{proof}

\begin{figure}[ht]
\centerline{(a)\includegraphics[width=0.28\textwidth]{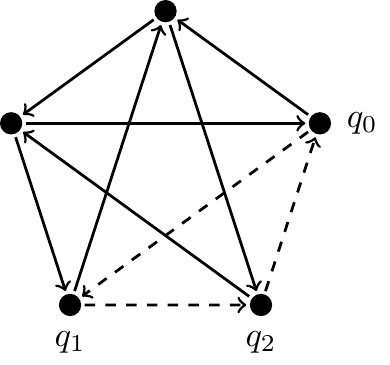}
\hspace{70pt}
(b)\includegraphics[width=0.28\textwidth]{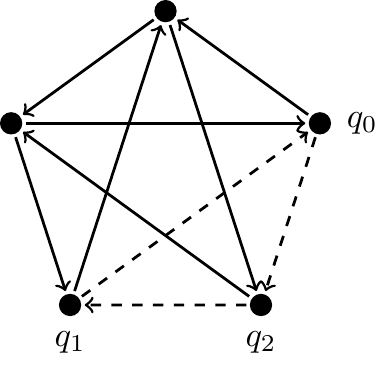}}
\caption{The Markov process consists of choosing uniformly at random one of the $d_N$ triangles in the regular tournament graph (a) and then reversing the orientation in order to obtain (b). This preserves the number of incoming and outgoing edges to all vertices, or, in terms of the corresponding adjacency matrix, this preserves the condition $\sum_{q} A_{pq} = (N-1)/2$ for all $p$.}
\label{Fig: RW for RITE}
\end{figure}

The random walk is performed by choosing one of these $d_N$ triangles $\bq$ uniformly at random and then reversing the orientation, i.e. $H_{q_0q_1}, H_{q_1q_2}, H_{q_2q_0} \to -H_{q_0q_1}, -H_{q_1q_2}, -H_{q_2q_0}$ (see Figure \ref{Fig: RW for RITE}). This guarantees the new matrix $H' = H + \delta H^{\bq}$ is contained in $\RT_N$ as it satisfies $\sum_{q} H'_{pq} = 0$ for all $p$. The difference matrix is given explicitly by
\begin{equation}\label{Eqn: RITE Matrix change}
\delta H^{\bq} := H' - H = \sum_{i=0}^2 (-2H_{q_iq_{i+1(3)}})(\e_{q_i}\e_{q_{i+1(3)}}^T - \e_{q_{i+1(3)}}\e_{q_i}^T ).
\end{equation}
We may summarise this random walk in the following transition probability for $H,H' \in \RT_N$
\begin{equation}\label{Eqn: RRTE motion}
\rho(H \to H') = \left\{\begin{array}{ll}
\frac{1}{d_N} & |H -H'|_{\RT_N} = 1 \\
0 & |H -H'|_{\RT_N} \neq 1, \end{array} \right.
\end{equation}
where $|H -H'|_{\RT_N} = \frac{1}{6}\sum_{p,q} |H_{pq} - H'_{pq}|$ is equal to 1 if and only if $H,H' \in \RT_N$ differ by the reversal of exactly one triangle. Starting at any tournament $H \in \R_N$, one may reach any other tournament $H' \in \R_N$ by performing successive reversals. Moreover, this Markov process is known to be mixing \cite{Kannan-1997}.

If $P_t(H)$ is the probability of the random walker to be at $H$ at time $t$ then
\[
P_{t+1}(H') = \sum_{H \in \RT_N} \rho(H \to H') P_t(H) = \sum_{H \in \RT_N : |H - H'|_{\RT_N} = 1} \frac{1}{d_N} P_t(H).
\]
Thus $P_t(H) = |\RT_N|^{-1}$ implies that $P_{t+1}(H') = |\RT_N|^{-1}$ also, i.e. $H$ and $H'$ are an exchangeable pair.

Using the indicator function (\ref{Eqn: RT Indicator}) the expected change of some observable under this random walk is therefore
\begin{equation}\label{Eqn: RITE expected change}
\E[\delta f |H] := \sum_{H' \in \RT_N} \rho(H \to H') [f(H') - f(H)] = \sum_{\bq} \frac{\Theta_{\bq}(H)}{d_N}[f(H + \delta H^{\bq}) - f(H)].
\end{equation}
Similarly, higher moments are obtained by taking the expectation of products of changes, i.e. for $f_1(H),f_2(H),\ldots,f_k(H)$
\begin{equation}\label{Eqn: RITE expected change multiple variables}
\E[\delta f_1 \ldots \delta f_k | H] := \sum_{\bq} \frac{\Theta_{\bq}(H)}{6d_N}[f_1(H + \delta H^{\bq}) - f_1(H)] \ldots [f_k(H + \delta H^{\bq}) - f_k(H)].
\end{equation}
Here we are again interested in the particular observables $Y_n(H)$ given in (\ref{Eqn: Centred Chebyshev}). Using the expression (\ref{Eqn: Chebyshev trace}) for $Y_n(H)$ in Section \ref{Sec: Graph theoretical tools} we find
\[
Y_n(H)  = \frac{1}{2} \frac{1}{(N-2)^n} \sum_{\omega \in \Omega_{2n}} H_\omega - \E[H_{\omega}] =  \frac{1}{2} \frac{1}{(N-2)^n} \sum_{\omega \in \Lambda_{2n}} H_\omega - \E[H_{\omega}].
\]
with $\Omega_{2n}$ and $\Lambda_{2n}$ the same as in previous sections. Note, however, that in contrast to the analogous expression (\ref{Eqn: Y formula ITE}) for the ITE the expectation term is not identically zero. This is precisely due to the global correlations enforced by demanding the row sums of $H$ are zero and will require the use of Lemma \ref{Lem: Regular tournament expectation} below to evaluate.

The following proposition describes how the $Y_n(H)$ behave under the aforementioned random walk.

\begin{proposition}\label{Prop: RITE evolution}Let $(H',H)$ be an exchangeable pair from the RITE connected via (\ref{Eqn: RRTE motion}). Let $Y_n(H)$ be as defined in (\ref{Eqn: Centred Chebyshev}) with $N$ sufficiently large. Then
\begin{enumerate}
\item \label{Eqn: RITE drift} $\frac{d_N}{6N}\E[\delta Y_n|H] = -nY_n(H) + R_n(H)$ (Drift term)
\item \label{Eqn: RITE diffusion}$\frac{d_N}{6N}\E[\delta Y_n\delta Y_m|H] = 2n^2\delta_{nm} + R_{nm}(H)$ (Diffusion term)
\item \label{Eqn: RITE remainder}$\frac{d_N}{6N}\E[|\delta Y_n\delta Y_m\delta Y_l | |H] = R_{nml}(H)$ (Remainder term)
\end{enumerate}
with $\E|R_n(H)| = \O(N^{-\frac{1}{2}})$, $\E|R_{nm}(H)| =\O(N^{-1})$ and $\E|R_{nml}(H)| = \O(N^{-1})$ for all $n,m,l = 2,\ldots,k$.
\end{proposition}

\begin{proof}
Parts \ref{Eqn: RITE drift}, \ref{Eqn: RITE diffusion} and \ref{Eqn: RITE remainder} of Proposition \ref{Prop: RITE evolution} will be proved in Sections \ref{Sec: RITE Drift}, \ref{Sec: RITE Diffusion} and \ref{Sec: RITE Remainder} respectively.
\end{proof}

Before progressing to Section \ref{Sec: RITE Drift} we first outline two lemmas that are necessary for the proofs.

\begin{lemma}\label{Lem: Indicator simplification}We have the simplification
\begin{equation}
\E[\delta f|H] = \sum_{\bq} \frac{\Theta_{\bq}(H)}{d_N} [f(H + \delta H^{(\bq)}) - f(H)]  = \sum'_{\bq}\frac{(1 - 3H_{q_0q_1}H_{q_1q_2})}{4 d_N}[f(H + \delta H^{(\bq)}) - f(H)],
\end{equation}
where the prime in the sum denotes that $q_0 \neq q_1 \neq  q_2 \neq q_0$ .
\end{lemma}
\begin{proof}Starting with the expression (\ref{Eqn: Indicator function}) for $\Theta_{\bq}(H)$ we can remove the factor $(1- \delta_{q_0q_1}\delta_{q_1q_2}\delta_{q_2q_0})$ provided we assume that $q_0 \neq q_1 \neq  q_2 \neq q_0$. Therefore, expanding in the same way as (\ref{Eqn: Indicator expansion}) and writing $\delta f^{\bq} := f(H + \delta H^{(\bq)}) - f(H)$ to condense notation we find
\begin{eqnarray}
\E[\delta f|H]  & = & \frac{1}{8 d_N}\sum'_{\bq}(1 - H_{q_0q_1}H_{q_1q_2})(1 - H_{q_1q_2}H_{q_2q_0})(1 - H_{q_2q_0}H_{q_0q_1})\delta f^{\bq} \nonumber \\
 & = & \frac{1}{4 d_N}\sum'_{\bq}(1  - H_{q_0q_1}H_{q_1q_2} - H_{q_1q_2}H_{q_2q_0} - H_{q_2q_0}H_{q_0q_1}) \delta f^{\bq} \nonumber \\
 & = & \frac{1}{4 d_N}\sum'_{\bq}(1  - 3H_{q_0q_1}H_{q_1q_2}) \delta f^{\bq},
\end{eqnarray}
where in the last line we have cyclicly permuted the indices in $q_0,q_1$ and $q_2$.
\end{proof}

\begin{lemma}\label{Lem: Regular tournament expectation} Let $\bp = (p_0,\ldots,p_{v-1})$ be $v$ distinct vertices and $E = \{(p_i,p_j)\}$ be a collection of $k$ edges on these vertices. Let us write $H_E := \prod_{(p,q) \in E} H_{pq}$ for the product of matrix elements over these edges and $\E$ the expectation over the RITE. Then
\begin{equation}\label{Eqn: RT expectation}
\E[H_E] = \O(N^{-\frac{k}{2}}).
\end{equation}
\end{lemma}
\begin{proof}See Appendix \ref{App: Expectation in RITE}.
\end{proof}

We stress the above lemma provides a key part in our subsequent analysis of the remainder terms in the motion of $Y_n(H)$. In contrast to Wigner ensembles, in which the elements are independent with mean 0, the RITE has global correlations between the matrix elements. However, Lemma \ref{Lem: Regular tournament expectation} shows that whilst we do not have $\E[H_E] = \prod_{(p,q) \in E}\E[H_{pq}] = 0$, as in the Wigner case, the correlations for a fixed number of elements are sufficiently weak as to allow for convergence to universal behaviour in the large $N$ limit.

\vspace{5pt}

Using the expression (\ref{Eqn: RITE Matrix change}), the $e = (p',q')$-th element of $\delta H^{\bq}$ is, in analogy to the corresponding ITE expression (\ref{Eqn: ITE matrix element change}), given by
\[
\delta H^{\bq}_e = -2\sum_{i=0}^2H_{q_iq_{i+1}}(\delta_{p'q_i}\delta_{q'q_{i+1}} - \delta_{p'q_i}\delta_{q'q_{i+1}})
= -2H_{e}\sum_{i=0}^2(\delta_{p'q_i}\delta_{q'q_{i+1}} +\delta_{p'q_i}\delta_{q'q_{i+1}}) = -2H_{e}\chi_{e,\bq}.
\]
Thus $\chi_{e,\bq}$ is an indicator function equal to 1 if the edge $e = (p',q')$ if it is equal to one of the (undirected) edges $\{(q_0,q_1),(q_1,q_0),(q_1,q_2)\}$ and 0 otherwise. The corresponding change in the matrix product $H_{\omega}$, for the non-backtracking cycle $\omega = (p_0,p_1,\ldots p_{2n-1},p_0)$ is therefore
\begin{equation}\label{Eqn: Homega change}
\delta H^{\bq}_{\omega}  := (H + \delta H^{\bq})_\omega  - H_{\omega} =  H_{\omega}\left[ \prod_{i=0}^{2n-1} (1 - 2\chi_{p_i p_{i+1(2n)},\bq}) - 1\right] = -2H_{\omega} \phi_{\omega,\bq},
\end{equation}
where we can also write
\[
\phi_{\omega,\bq} = \frac{1}{2}(1 - (-1)^{\chi_{\omega,\bq}}), \qquad \chi_{\omega,\bq} = \sum_{i= 0}^{2n-1} \chi_{p_i p_{i+1(2n)},\bq}.
\]
Thus $\phi_{\omega,\bq}$ is an indicator function equal to 1 if the edges in the triangle $\bq$ are traversed an odd number of times by $\omega$ and 0 otherwise.

The change in $Y_n(H)$ brought about by reversing the orientation of $\bq$ can therefore be expressed using (\ref{Eqn: Homega change}) as
\begin{equation}\label{Eqn: deltaY RITE}
\delta Y^{\bq}_n := Y_n(H + \delta H^{\bq}) - Y_n(H)
= \frac{1}{2}\frac{1}{(N-2)^n}\sum_{\omega \in \Lambda_{2n}} \delta H^{\bq}_{\omega} 
 = -\frac{1}{(N-2)^n}\sum_{\omega \in \Lambda_{2n}} H_{\omega}\phi_{\omega,\bq}.
\end{equation}

\subsection{Proof of Proposition \ref{Prop: RITE evolution} Part \ref{Eqn: RITE drift} - Drift term}\label{Sec: RITE Drift}
Starting from the expression (\ref{Eqn: RITE expected change}) for the expected change of an observable, inserting the expression (\ref{Eqn: deltaY RITE}) and utilising Lemma \ref{Lem: Indicator simplification} we find
\begin{equation}\label{Eqn: Expected Y change RITE}
\E[\delta Y_n| H] := \frac{1}{ d_N}\sum_{\bq} \Theta_{\bq}(H) \delta Y^{\bq}_n  =  - \frac{1}{2}\frac{1}{(N-2)^n}  \frac{1}{2 d_N} \sum_{\omega \in \Lambda_{2n}} \sum'_{\bq}(1 - 3H_{q_0q_1}H_{q_1q_2}) H_{\omega} \phi_{\omega,\bq}.
\end{equation}
Therefore we may write
\[
\E[\delta Y_n| H] = \frac{6 N}{d_N}[-nY_n(H)  + R_n(H)],
\]
with the remainder given by
\begin{equation}\label{Eqn: Drift remainder expression 1}
R_n(H) = \frac{1}{2}\frac{n}{(N-2)^n}\sum_{\omega \in \Lambda_{2n}}\bigg[ H_{\omega}\bigg( 1 - \frac{1}{12nN}\sum_{\bq}'\phi_{\omega,\bq}\bigg) - \bigg(\E[H_{\omega}] -  \frac{1}{4nN}\sum_{\bq}'\phi_{\omega,\bq} H_{q_0q_1}H_{q_1q_2}H_{\omega}\bigg)  \bigg].
\end{equation}
Now, crucially, by splitting the sum over $\Lambda_{2n}$ into $\Lambda^{\star}_{2n} = \{\omega \in \Lambda_{2n} : |F_{\omega}| = 2n\}$ and $\Lambda^{\circ}_{2n} = \Lambda_{2n} \setminus \Lambda^{\star}_{2n}$ (see Section \ref{Sec: ITE}) the constant expectation term in the above can be expressed, subject to a subleading correction in $N$, in the following alternative manner
\begin{align}
\sum_{\omega \in \Lambda_{2n}}\E[H_{\omega}] & = \sum_{\omega \in \Lambda^{\star}_{2n}}\E[H_{\omega}]  + \sum_{\omega \in \Lambda^{\circ}_{2n}}\E[H_{\omega}]  \label{Eqn: Simplification 1} \\
& = \frac{1}{12nN}\sum_{\omega \in \Lambda^{\star}_{2n}} \sum_{\bq}' \phi_{\omega,\bq}  \E[H_{\omega}] + \O(N^{n-1}) \label{Eqn: Simplification 2}\\ 
& = \frac{1}{12nN}\sum_{\omega \in \Lambda_{2n}} \sum_{\bq}'\phi_{\omega,\bq}  \E[H_{\omega}]+ \O(N^{n-1})  \label{Eqn: Simplification 3} \\ 
& =  \frac{1}{4nN}\sum_{\omega \in \Lambda_{2n}}\sum_{\bq}'\phi_{\omega,\bq} \E[H_{q_0q_1}H_{q_1q_2}H_{\omega}] + \O(N^{n-1})  \label{Eqn: Simplification 4} .
\end{align}
To see why this is the case, first note that Lemma \ref{Lem: Regular tournament expectation} implies $\sum_{\omega \in \Lambda^{\circ}_{2n}}\E[H_{\omega}] = \O(N^{\Phi})$, where $\Phi = \max_{\omega \in \Lambda^{\circ}_{2n}}\{|V_{\omega}| - |F_{\omega}|/2\}$ (see Section \ref{Sec: Graph theoretical tools} for the definition of $F_{\omega}$ - the set of free edges), with the contribution of $\O(N^{|V_{\omega}}|)$ coming from the number of possibilities of labelling the vertices in $\omega$. Let us consider those $\omega$ in which every edge is traversed at most twice (all other cycles will give a negligible contribution in comparison) and form the graph $\hat{G} = (V_{\omega},E_{\omega})$. Since $\omega$ is a cycle the graph $\hat{G}$ is connected and satisfies $2|E_{\omega}| - |F_{\omega}| = 2n$, with the first Betti number $\beta(\hat{G}) = |E_{\omega}| - |V_{\omega}| + 1$. Thus $|V_{\omega}| - |F_{\omega}|/2 = n +1 - \beta(\hat{G})$. Now $\beta(\hat{G}) >0$, otherwise the $\omega$ would be backtracking. In addition, suppose $\beta(\hat{G}) = 1$, then $\hat{G}$ must be a loop (there can be no dangling edges since $\omega$ is non-backtracking), however this is only possible for walks $\omega$ in which $|F_{\omega}| = 2n$ or $|F_{\omega}| = 0$, which means $\omega \notin \Lambda^{\circ}_{2n}$. Hence $\beta(\hat{G}) \geq 2$ and therefore $|V_{\omega}| - |F_{\omega}|/2 = n - 1$, meaning the second term in (\ref{Eqn: Simplification 1}) is of order $\O(N^{n-1})$.

In addition, for all $\omega \in \Lambda^{\star}_{2n}$ we have
\begin{equation}\label{Eqn: triangle summation 1}
\sum'_{\bq} \phi_{\omega,\bq} = 12n N  + \O(1).
\end{equation}
This comes from counting all triangles $\bq$ that share a single edge with $\omega$. If we fix, for instance, $(q_0,q_1) = (p_0,p_1)$ (the first edge in $\omega$) then there are $N + \O(1)$ possible values for $q_2$ for which $\phi_{\omega,\bq} = 1$. Noting there are 6 possible orientations of $\bq$ for each edge of $\omega$ and $2n$ edges gives (\ref{Eqn: triangle summation 1}). Moreover, for all $\omega \in \Lambda^{\star}_{2n}$ we have $|F_{\omega}| = 2n$ and $|V_{\omega}| \leq 2n$, so $V_{\omega} - |F_{\omega}|/2 \leq n$ and thus $\sum_{\omega \in \Lambda^{\star}_{2n}} \E[H_{\omega}] = \O(N^n)$. Combing this with (\ref{Eqn: triangle summation 1}) lead to (\ref{Eqn: Simplification 2}).

Applying the same reasoning we deduce that $\frac{1}{12nN}\sum'_{\bq} \phi_{\omega,\bq} = \O(1)$ for all $\omega \in \Lambda^{\circ}_{2n}$, allowing us to extend the sum in (\ref{Eqn: Simplification 3}). Finally, since $\E[\delta Y_n] = \E[\E[\delta Y_n | H]] = 0$, taking the expectation in (\ref{Eqn: Expected Y change RITE}) gives the result in (\ref{Eqn: Simplification 4}). Inserting this into the remainder (\ref{Eqn: Drift remainder expression 1}) therefore gives
\begin{equation}
R_n(H) = \frac{1}{2}\frac{n}{(N-2)^n}[ S^{(1)}(H) + S^{(2)}_n(H)] + \O(N^{-1}),
\end{equation}
where
\begin{equation}\label{Eqn: First part expression}
S^{(1)}_n(H) = \sum_{\omega \in \Lambda_{2n}}H_{\omega}\bigg( 1 - \frac{1}{12nN}\sum_{\bq}'\phi_{\omega,\bq}\bigg)
\end{equation}
and
\begin{equation}\label{Eqn: Second part expression}
S^{(2)}_n(H) = \frac{1}{4nN}\sum_{\omega \in \Lambda_{2n}} \sum_{\bq}'\phi_{\omega,\bq} \bigg(H_{q_0q_1}H_{q_1q_2}H_{\omega} - \E[H_{q_0q_1}H_{q_1q_2}H_{\omega}]\bigg).
\end{equation}
Part \ref{Eqn: RITE drift} of Proposition \ref{Prop: RITE evolution} thus follows immediately from the triangle equality and the following lemma
\begin{lemma}\label{Lem: Drift remainder pre estimates} Let $S^{(1)}_n(H)$ and $S^{(2)}_n(H)$ be as defined in (\ref{Eqn: First part expression}) and (\ref{Eqn: Second part expression}) respectively and $\E$ denote the expectation of the RITE. Then
\begin{enumerate}
\item \label{Eqn: First part estimate} $\E|S^{(1)}(H)| = \O(N^{n-1})$
\item \label{Eqn: Second part estimate} $\E|S^{(2)}(H)| = \O(N^{n-1/2})$.
\end{enumerate}
\end{lemma}
\begin{proof}[Proof of Lemma \ref{Lem: Drift remainder pre estimates} Part \ref{Eqn: First part estimate}] Firstly, for notational convenience, let us write $\kappa_{\omega} := (1 - \frac{1}{12nN}\sum'_{\bq} \phi_{\omega,\bq})$. Splitting the sum over $\Lambda_{2n}$ in (\ref{Eqn: First part expression}) into $\Lambda^{\star}_{2n}$ and $\Lambda^{\circ}_{2n}$ leads to
\begin{multline}\label{Eqn: First part expectation expression}
\E|S^{(1)}_n(H)| \leq \sum_{[\omega] \in [\Lambda^{\star}_{2n}] } |\kappa_{[\omega]}|  \E\bigg|\sum_{\omega \in [\omega] } H_{\omega}  \bigg| +  \sum_{[\omega] \in [\Lambda^{\circ}_{2n}] } |\kappa_{[\omega]}|  \E\bigg|\sum_{\omega \in [\omega] } H_{\omega} \bigg| \\
\leq  \sum_{[\omega] \in [\Lambda^{\star}_{2n}] } |\kappa_{[\omega]}| \sqrt{\sum_{\omega,\omega' \in [\omega] } \E[H_{\omega,\omega'}]} +  \sum_{[\omega] \in [\Lambda^{\circ}_{2n}] } |\kappa_{[\omega]}| \sqrt{\sum_{\omega,\omega' \in [\omega] } \E[H_{\omega,\omega'}]},
\end{multline}
where we have used that $\kappa_{\omega} = \kappa_{\omega'}$ for all $\omega \sim \omega'$. For a particular equivalence class $[\omega]$, if $\Phi_{[\omega]} = \max_{\omega,\omega' \in [\omega]}\{ |V_{\omega,\omega'}| - |F_{\omega,\omega'}|/2$\} then, using Lemma \ref{Lem: Regular tournament expectation}, the quantity $\sum_{\omega,\omega' \in [\omega] } \E[H_{\omega,\omega'}]$ is of order $\O(N^{\Phi_{[\omega]}})$. 

For $\omega \in \Lambda^{\star}_{2n}$ we have $|F_{\omega}| = 2n$, meaning the graph $G = (V_{\omega},F_{\omega})$ will have one connected component and $\beta(G) \geq 1$. Therefore, by Corollary \ref{Cor: RITE vertex edge identity}, we find that for $\omega \sim \omega' \in [\Lambda^{\star}_{2n}]$, $|V_{\omega,\omega'}| - |F_{\omega,\omega'}|/2 \leq |V_{\omega}| \leq 2n$. In addition (\ref{Eqn: triangle summation 1}) implies that $\kappa_{[\omega]} = \O(1)$. Hence the first term in (\ref{Eqn: First part expectation expression}) is of order $\O(N^{-1})\sqrt{\O(N^{2n})} = \O(N^{n-1})$. 

For $\omega \in \Lambda^{\circ}_{2n}$ we have $0 < |F_{\omega}| < 2n$, which implies the graph $G = (V_{\omega},F_{\omega})$ will have multiple connected components, which we can label $i=1,\ldots,C$. However, since $\omega$ is a cycle, those components satisfying $\beta_i = 0$ must be isolated vertices. Let us suppose that all edges in $\omega$ are traversed a maximum of twice (more than twice will give lower order contributions), then the graph $\hat{G} = (V_{\omega},E_{\omega}) \supseteq G$ must be connected and satisfy $2|E_{\omega}| - |F_{\omega}| = 2n$. Now, if $|V_I| = \sum_i \delta_{\beta_i,0}$ counts the number of isolated vertices then we must have $|V_I| \leq |E_{\omega}|- |F_{\omega}| - 1$, since for $\omega \in \Lambda^{\circ}_{2n}$ the number of edges traversed twice (given by $|E_{\omega}|- |F_{\omega}|$) must be at least one more than the number of isolated vertices. Hence, $|V_I| \leq 2n - |E_{\omega}| - 1 = 2n - |V_{\omega}| - \beta(\hat{G})$. Thus using Corollary \ref{Cor: RITE vertex edge identity}, we have $|V_{\omega,\omega'}| - |F_{\omega,\omega'}|/2 \leq |V_{\omega}| + |V_I| = 2n - \beta(\hat{G}) \leq 2n - 2$ because $\beta(\hat{G}) \geq 2$ for all $\omega \in \Lambda^{\circ}_{2n}$. In addition $\kappa_{\omega} = O(1)$ for $\omega \in \Lambda^{\circ}_{2n}$ so the second term in (\ref{Eqn: First part expectation expression}) is of order $\O(1)\sqrt{\O(N^{2n-2})}  = \O(N^{n-1})$.
\end{proof}

\begin{proof}[Proof of Lemma \ref{Lem: Drift remainder pre estimates} Part \ref{Eqn: Second part estimate}]
Let us define the following sets of walks
\begin{align}
A_{r} & = \{(p_0,\ldots,p_{r-1},p_0,q,p_{r-1}): p_i \ \text{distinct}, q \neq p_0,p_1,p_{r-2},p_{r-1}\} \label{Eqn: Walks A dfn}\\
B_{r} & = \{(p_0,\ldots,p_{r-1},p_1): p_i \ \text{distinct} \} \label{Eqn: Walks B dfn} \\
C_{r} & = \{(p_0,\ldots,p_{r-1}): p_i \ \text{distinct} \} \label{Eqn: Walks C dfn} \\
D_{r} & = \{(p_0,\ldots,p_{r-1},q): p_i \ \text{distinct}, q= p_0,\ldots,p_{r-3}\} \label{Eqn: Walks D dfn}
\end{align}

\begin{figure}[ht]
\centerline{\includegraphics[width=0.2\textwidth]{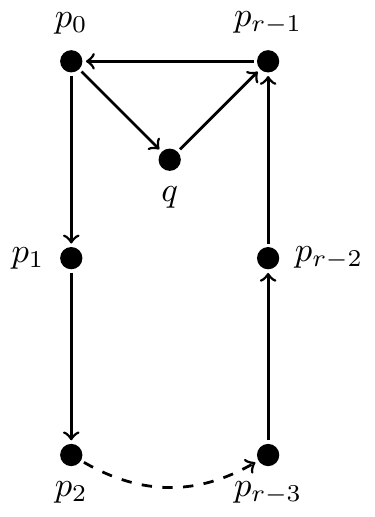}
\hspace{20pt}
\includegraphics[width=0.2\textwidth]{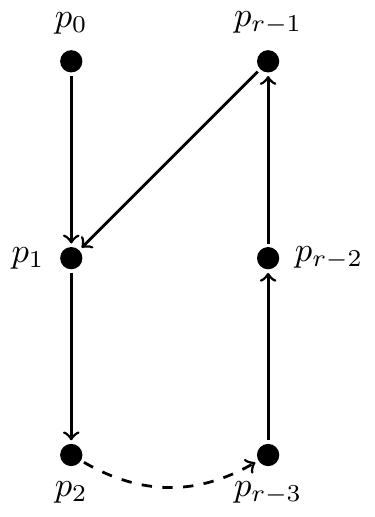}
\hspace{20pt}
\includegraphics[width=0.2\textwidth]{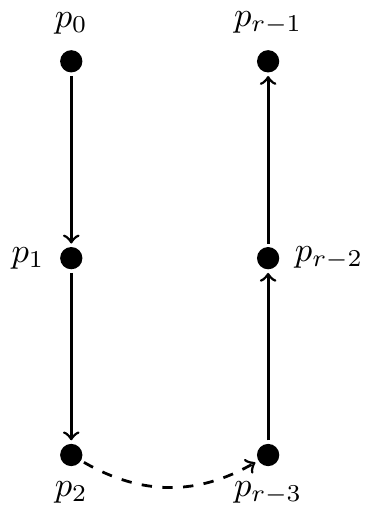}
\hspace{20pt}
\includegraphics[width=0.2\textwidth]{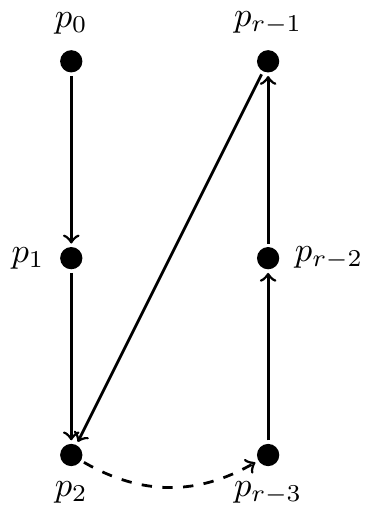}}
\centerline{(a) \hspace{100pt} (b) \hspace{100pt} (c) \hspace{100pt} (d)}
\caption{Examples of walks in (a) $A_r$, where $q$ may also be equal to $p_2,p_3,\ldots,p_{r-4}$, (b) $B_r$, (c) $C_r$ and (d) $D_r$, where the last vertex in the walk may be any of $p_0,\ldots,p_{r-3}$.}
\label{Fig: Examples of walks}
\end{figure}

\begin{proposition}\label{Prop: Second part decomp} Define $\Lambda^{\dagger}_{2n} : = \{\omega \in \Lambda_{2n} : |V_{
\omega}|  = 2n\}$ and $\Lambda^{\times}_{2n} = \Lambda_{2n} \setminus \Lambda^{\dagger}_{2n}$. Then, splitting the sum over $\Lambda_{2n}$ in (\ref{Eqn: Second part expression}) leads to
\begin{multline}\label{Eqn: Prop 1 statement}
S^{(2)}_n(H) = \frac{1}{4nN}\bigg[2n\sum_{\omega \in A_{2n}}H_{\omega} - 4n\sum_{\omega \in \Lambda^{\dagger}_{2n}}H_{\omega} - 4n\sum_{\omega \in B_{2n}} H_{\omega} 
+ 4n\sum_{j=1}^{n-2} \O(N^{j})\sum_{\omega \in D_{2n-2j}}H_{\omega} \\ 
+ \sum_{\omega \in \Lambda^{\times}_{2n}} \sum_{\bq}'\phi_{\omega,\bq} H_{q_0q_1}H_{q_1q_2}H_{\omega} \bigg] + \O(N^{n-1}).
\end{multline}
\end{proposition}
Given Proposition \ref{Prop: Second part decomp} we have by the triangle inequality
\begin{multline}\label{Eqn: S2 triangle inequality}
\E|S^{(2)}_n(H)| \leq \O(N^{-1})\bigg[  \E\bigg|\sum_{\omega \in A_{2n}}H_{\omega}\bigg| + \E\bigg|\sum_{\omega \in \Lambda^{\dagger}_{2n}}H_{\omega}\bigg|  + \E\bigg|\sum_{\omega \in B_{2n}} H_{\omega} \bigg|
+ \sum_{j=1}^{n-2} \O(N^{j}) \E\bigg|\sum_{\omega \in D_{2n-2j}}H_{\omega}\bigg| \\ + \E\bigg|\sum_{\omega \in \Lambda^{\times}_{2n}} \sum_{\bq}'\phi_{\omega,\bq} H_{q_0q_1}H_{q_1q_2}H_{\omega}\bigg| \bigg]+ \O(N^{n-1}).
\end{multline}
The result is then obtained by showing all the terms within the square brackets are at most of order $\O(N^n)$.

We start with walks $\omega \in A_{2n}$. As before, we note that
\begin{equation}\label{Eqn: A walks contribution}
\E\bigg|\sum_{\omega \in A_{2n}}H_{\omega}\bigg| \leq \sum_{[\omega] \in [A_{2n}]} \sqrt{\sum_{\omega,\omega' \in [\omega]}\E[H_{\omega,\omega'}]}.
\end{equation}
From Corollary \ref{Cor: RITE vertex edge identity} we have, for all $\omega \sim \omega' \in A_{2n}$, that $|V_{\omega,\omega'}| - |F_{\omega,\omega'}|/2 \leq |V_{\omega}| \leq 2n+1$, giving a contribution of order $\sqrt{\O(N^{2n+1})} = \O(N^{n + 1/2})$. Similarly, taking the same inequality for walks $\omega \in D_{2n-2j}$ and noting that $|V_{\omega,\omega'}| - |F_{\omega,\omega'}|/2 \leq |V_{\omega}| = 2n - 2j$ for all $\omega \sim \omega' \in D_{2n-2j}$ leads to a contribution $\E|\sum_{\omega \in D_{2n-2j}}H_{\omega} | = \O(N^{n-j})$. Finally the inclusions $B_{2n},\Lambda^{\dagger}_{2n} \subset D_{2n}$, immediately imply the respective terms in (\ref{Eqn: S2 triangle inequality}) are also of order $O(N^{n})$.

\vspace{10pt}

It thus remains to estimate the term involving walks in $\Lambda^{\times}_{2n}$. Let us define $\tilde{\Lambda}^{\times}_{2n} := \{(\bq,\omega) : q_0 \neq q_1 \neq q_2 \neq q_0, \omega \in \Lambda^{\times}_{2n}, \phi_{\omega,\bq} =1 \}$ and $H_{\bq,\omega} := H_{q_0q_1}H_{q_1q_2}H_{\omega}$. Examples of (collections of) walks in $\tilde{\Lambda}^{\times}_{2n}$ are given in Figure \ref{Fig: Examples of other walks}. The final term in (\ref{Eqn: S2 triangle inequality}) is less than or equal to
\begin{equation}\label{Eqn: Final contr of second part decomp}
\sum_{[\bq,\omega] \in [\tilde{\Lambda}^{\times}_{2n}]} \sqrt{\sum_{(\bq,\omega), (\bq',\omega') \atop \in [\bq,\omega]} \E[H_{\bq,\omega}H_{\bq',\omega'}]}.
\end{equation}
If we write $V_{\bq} = \{q_0,q_1,q_2\}$ and $E_{\bq} = \{(q_0,q_1),(q_1,q_2)\}$ then $V_{\bq,\omega}$, $E_{\bq,\omega}$ and $F_{\bq,\omega}$ are defined in the usual manner. Let us define the graph $\hat{G} = (V_{\bq,\omega},E_{\bq,\omega})$ and assume that edges in $E_{\bq,\omega}$ are traversed a maximum of twice, meaning $2|E_{\bq,\omega}| - |F_{\bq,\omega}| = 2n + 2$. Note that $\hat{G}$ must be connected. Let $V_I$ be the set of isolated vertices in $G = (V_{\bq,\omega},F_{\bq,\omega})$. For example, in Figure \ref{Fig: Examples of other walks} (a) $V_I = \{p_{2n-1}\}$ since edges $(p_{2n-1},p_0), (p_1,p_{2n-1}) \notin F_{\bq,\omega}$, whereas there are no isolated vertices in Figure \ref{Fig: Examples of other walks} (b) and (c).

The condition $\phi_{\omega,\bq} =1 $ implies that $\omega$ and $\bq$ must share an odd number of edges, i.e. $|E_{\omega} \cap \{(q_0,q_1),(q_1,q_2),(q_2,q_0)\}| = 1 \mod(2)$. This leads to two scenarios: Either $|E_{\omega} \cap E_{\bq}| \neq \emptyset$ (see e.g Figure \ref{Fig: Examples of other walks} (a) and (c)) or $E_{\omega} \cap E_{\bq} = \emptyset$ (see e.g. Figure \ref{Fig: Examples of other walks} (b)).

In the first scenario, since at least one edge in $E_{\bq,\omega}$ must be traversed twice, the number of isolated vertices satisfies $|V_I| \leq |E_{\bq,\omega}| - |F_{\bq,\omega}| - 1$.
Therefore, $|V_I| \leq 2n + 2 - |E_{\bq,\omega}| - 1 = 2n + 2 - |V_{\bq,\omega}| - \beta(\hat{G})$. By Corollary \ref{Cor: RITE vertex edge identity} we therefore have 
$|V_{\bq,\omega,\bq',\omega'}| - |F_{\bq,\omega,\bq',\omega'}|/2 \leq 2n + 2 - \beta(\hat{G}) \leq 2n$, since $\omega \in \Lambda^{\times}_{2n}$ implies that the subgraph traced out by $\omega$ must have at least two fundamental cycles and thus $\beta(\hat{G}) \geq 2$. 

In the second scenario it is possible that all edges are free (i.e. traversed only once) in $(\bq,\omega)$ meaning $|E_{\bq,\omega}| - |F_{\bq,\omega}| \geq 0$. The number of isolated vertices must therefore satisfy $|V_I| \leq |E_{\bq,\omega}| - |F_{\bq,\omega}|$, which, via Corollary \ref{Cor: RITE vertex edge identity}, leads to the inequality $|V_{\bq,\omega,\bq',\omega'}| - |F_{\bq,\omega,\bq',\omega'}|/2  \leq 2n + 3 - \beta(\hat{G})$. However, in contrast to the first scenario, the union of $E_{\bq}$ and $E_\omega$ means additional fundamental cycles must be added, i.e. $\beta(\hat{G}) \geq 3$, implying that $|V_{\bq,\omega,\bq',\omega'}| - |F_{\bq,\omega,\bq',\omega'}|/2 \leq 2n$ again. Thus (\ref{Eqn: Final contr of second part decomp}) and, in turn, the final term in (\ref{Eqn: S2 triangle inequality}) is of order $\O(N^n)$.

\begin{figure}[ht]
\centerline{\includegraphics[width=0.22\textwidth]{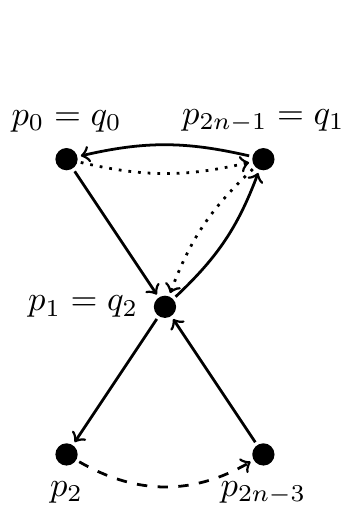}
\hspace{30pt}
\includegraphics[width=0.22\textwidth]{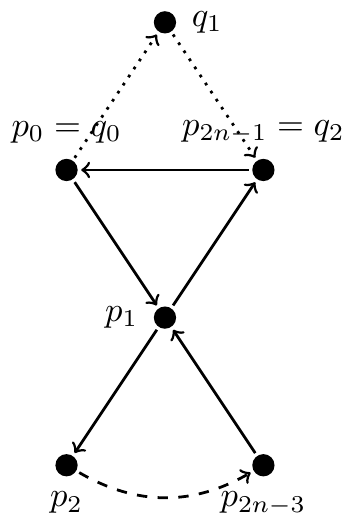}
\hspace{30pt}
\includegraphics[width=0.22\textwidth]{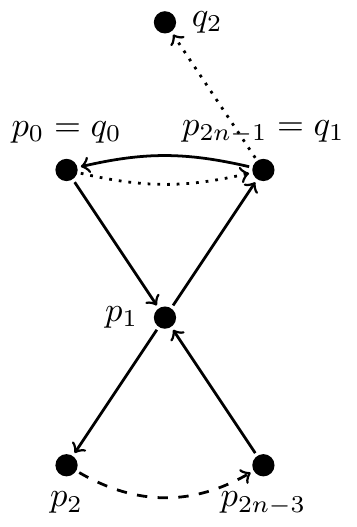}}
\centerline{(a) \hspace{120pt} (b) \hspace{120pt} (c) \hspace{10pt}}
\caption{Examples of walks $(\bq,\omega) \in \tilde{\Lambda}^{\times}_{2n}$ in which there are either (a) 3 edges of $\bq$ contained in $\omega$ and (b) and (c) 1 edge of $\bq$ contained in $\omega$.}
\label{Fig: Examples of other walks}
\end{figure}

\end{proof}

\begin{proof}[Proof of Proposition \ref{Prop: Second part decomp}]The idea will be to split the sum over $\Lambda_{2n}$ into those sets $\Lambda^{\dagger}_{2n}$ and $\Lambda^{\times}_{2n}$, then show the sum over $\Lambda^{\dagger}_{2n}$ can be re-expressed in the form of the first four terms in (\ref{Eqn: Prop 1 statement}), up to a correction of $\O(N^{-1})$. Let us start with the first term in (\ref{Eqn: Second part expression}), since $|V_{\omega}| = 2n$ for all $\omega \in \Lambda^{\dagger}_{2n}$, the condition $\phi_{\omega,\bq} =1$ means that $\bq$ and $\omega$ must share a single edge. Therefore writing out explicitly all those terms in which the edge $(p_i,p_{i+1(2n)}) = (q_j, q_{j+1(3)}$ or $(q_{j+1(3)}, q_j)$ for $i=0,\ldots, 2n-1$, $j =0,1,2$ gives
\begin{multline}\label{Eqn: Prop 1 first sum expression}
\sum_{\omega \in \Lambda^{\dagger}_{2n}}\sum_{\bq}'\phi_{\omega,\bq}H_{q_0q_1}H_{q_1q_2}H_{\omega} = \sum'_{p_0,\ldots,p_{2n-1}} H_{\omega} \sum_{i=0}^{2n-1}\sum_{q \notin \mathcal{P}_i} (H_{p_i q}H_{qp_{i+1(2n)}} + H_{p_{i+1(2n)}q}H_{qp_i})   \\
+  \sum'_{p_0,\ldots,p_{2n-1}} H_{\omega} \sum_{i=0}^{2n-1}\sum_{q \notin \mathcal{P}_i} (H_{p_i p_{i+1(2n)}}H_{p_{i+1(2n)}q} + H_{p_{i+1(2n)}p_i }H_{p_{i}q} + H_{qp_i} H_{p_i p_{i+1(2n)}} + H_{q p_{i+1(2n)}} H_{p_{i+1(2n)} p_i} )  \\
= 4n\sum_{\omega \in A_{2n}}H_{\omega} + 8n\sum_{p_0,\ldots,p_{2n-1}}' \sum_{q \neq p_0,p_1,p_{2n-2}} H_{p_0p_1} \ldots H_{p_{2n-2}p_{2n-1}}H_{p_{2n-1}q},
\end{multline}
where $\mathcal{P}_i = \{p_{i-1(2n)},p_i,p_{i+1(2n)},p_{i+2(2n)}\}$ and $A_{2n}$ is defined in (\ref{Eqn: Walks A dfn}).

The second term in (\ref{Eqn: Prop 1 first sum expression}) may be further modified by using the regularity $H$, i.e.
\begin{multline}\label{Eqn: Extended sum over p and q}
\sum_{p_0,\ldots,p_{2n-1}}' \sum_{q \neq p_0,p_1,p_{2n-3}} H_{p_0p_1} \ldots H_{p_{2n-2}p_{2n-1}}H_{p_{2n-1}q} \\
= - \sum_{p_0,\ldots,p_{2n-1}}' H_{p_0p_1} \ldots H_{p_{2n-2}p_{2n-1}}(H_{p_{2n-1}p_0}  + H_{p_{2n-1}p_1} + H_{p_{2n-1}p_{2n-2}}) \\
= - \sum_{\omega \in \Lambda^{\dagger}_{2n}} H_{\omega} - \sum_{\omega \in B_{2n}}H_{\omega} - (N-(2n-1))\sum_{\omega \in C_{2n-1}}H_{\omega},
\end{multline}
with $B_{2n}$ and $C_{2n-1}$ given in (\ref{Eqn: Walks A dfn}) and (\ref{Eqn: Walks B dfn}) respectively. The summation of walks in $C_{2n-1}$ is obtained by noting that we have $H_{p_{2n-2}p_{2n-1}}H_{p_{2n-1}p_{2n-2}} = 1$ and summing over the free variable $p_{2n} \neq p_0,\ldots, p_{2n-2}$.

To proceed further we note that a similar application of the regularity condition can be applied to the index $p_{2n-2}$ 
\begin{multline}
\sum_{\omega \in C_{2n-1}}H_{\omega} := \sum_{p_0,\ldots,p_{2n-3}}'\sum_{p_{2n-2} \neq p_0,\ldots,p_{2n-3}} H_{p_0p_1}\ldots H_{p_{2n-4}p_{2n-3}}  H_{p_{2n-3}p_{2n-2}} \\
= - \sum'_{p_0,\ldots,p_{2n-3}}\sum_{q = p_0,\ldots,p_{2n-4}} H_{p_0p_1}\ldots H_{p_{2n-3}q} - \sum_{p_0,\ldots,p_{2n-3}}' H_{p_0p_1}\ldots H_{p_{2n-4}p_{2n-3}}H_{p_{2n-3}p_{2n-4}} \\
= - \sum_{\omega \in D_{2n-2}} H_{\omega} - (N-(2n-3))\sum_{\omega \in C_{2n-3}}H_{\omega},
\end{multline}
with $D_{2n-2}$ given in (\ref{Eqn: Walks D dfn}).

Applying the same method to $C_{2n-3}$ and so forth leads to 
\begin{multline}
\sum_{\omega \in C_{2n-1}}H_{\omega} = -\sum_{\omega \in D_{2n-2}} H_{\omega} + \sum_{j=2}^{n-2}(-1)^j \left(\prod_{r=1}^{j-1}(N-(2n-1) - 2r)\right)\sum_{\omega \in D_{2n-2j}}H_{\omega} \\ 
+ \  (-1)^n \left(\prod_{r=1}^{n-2}(N- (2n-1) - 2r)\right)\sum_{\omega \in C_3} H_{\omega}.
\end{multline}
Moreover, due to the regularity of $H$, the sum over $C_3$ is constant
\[
\sum_{\omega \in C_3} H_{\omega} = \sum'_{p_0,p_1}\sum_{p_2 \neq p_0,p_1} H_{p_0p_1}H_{p_1p_2} = \sum'_{p_0,p_1} (-1) = -N(N-1),
\]
which gives
\begin{equation}\label{Eqn: Sum over C_2n-1}
\sum_{\omega \in C_{2n-1}}H_{\omega} =  \sum_{j=1}^{n-2} \O(N^{j-1}) \sum_{\omega \in D_{2n-2j}}H_{\omega} 
- \  (-1)^n N\left(\prod_{r=1}^{n-1}(N- (2n-1) - 2r)\right).
\end{equation}
Therefore, inserting (\ref{Eqn: Sum over C_2n-1}) into (\ref{Eqn: Extended sum over p and q}) and then (\ref{Eqn: Extended sum over p and q}) into (\ref{Eqn: Prop 1 first sum expression}) leads to the following expression
\begin{multline}\label{Eqn: Prop 1 basic quantity}
\sum_{\omega \in \Lambda^{\dagger}_{2n}}\sum_{\bq}'\phi_{\omega,\bq}H_{q_0q_1}H_{q_1q_2}H_{\omega} = 2n\sum_{\omega \in A_{2n}}H_{\omega} - 4n\sum_{\omega \in \Lambda^{\dagger}_{2n}}H_{\omega} - 4n\sum_{\omega \in B_{2n}} H_{\omega} \\
+ 4n(N-(2n-1))\left[\sum_{j=1}^{n-2} \O(N^{j-1})\sum_{\omega \in D_{2n-2j}}H_{\omega} + c_{n,N}\right],
\end{multline}
with the constant $c_{n,N}$ given by the second term in the right hand side of (\ref{Eqn: Sum over C_2n-1}). Importantly, this constant is of order $\O(N^n)$, which would lead to a larger result in Lemma \ref{Lem: Drift remainder pre estimates} Part \ref{Eqn: Second part estimate}, however by subtracting the expectation of the same quantity, as in (\ref{Eqn: Second part expression}) this leading order is removed. Hence, inserting (\ref{Eqn: Prop 1 basic quantity}) into (\ref{Eqn: Second part expression}) gives
\begin{multline}
S_2(H) = \frac{1}{4nN}\bigg[2n\sum_{\omega \in A_{2n}}H_{\omega} - 4n\sum_{\omega \in \Lambda^{\dagger}_{2n}}H_{\omega} - 4n\sum_{\omega \in B_{2n}} H_{\omega} 
+ 4n \sum_{j=1}^{n-2} \O(N^{j})\sum_{\omega \in D_{2n-2j}}H_{\omega}  \\
 - 2n\sum_{\omega \in A_{2n}}\E[H_{\omega}] + 4n\sum_{\omega \in \Lambda^{\dagger}_{2n}}\E[H_{\omega}]  + 4n\sum_{\omega \in B_{2n}} \E[H_{\omega}] 
- 4n \sum_{j=1}^{n-2} \O(N^{j})\sum_{\omega \in D_{2n-2j}}\E[H_{\omega}] \\
+\sum_{\omega \in \Lambda^{\times}_{2n}}\sum_{\bq}'\phi_{\omega,\bq}H_{q_0q_1}H_{q_1q_2}H_{\omega}  - \sum_{\omega \in \Lambda^{\times}_{2n}}\sum_{\bq}'\phi_{\omega,\bq}\E[H_{q_0q_1}H_{q_1q_2}H_{\omega}] \bigg].
\end{multline}
The result is therefore obtained once we show all terms involving expectations are at most $\O(N^n)$. We start with $D_{2n-2j}$. In this case each walk $\omega$ has $|V_{\omega}| = |F_{\omega}| = 2n-2j$ and therefore $|V_{\omega}| - |F_{\omega}|/2 = n-j$. Hence
\[
\sum_{\omega \in D_{2r}}\E[H_{\omega} ] = \O(N^{n-j}).
\]
The same holds for $B_{2n}$ and $\Lambda^{\dagger}_{2n}$ since they are both contained in $D_{2n}$. For walks in $A_{2n}$ we have $|V_{\omega}| \leq 2n+1$ and $|F_{\omega}| = 2n+2$, giving $|V_{\omega}| - |F_{\omega}|/2 \leq n$ and so the same result follows.

In the final term the walk $\omega$ must share at least one edge with $\bq$ due to the condition $\phi_{\omega,\bq}$. As in the proof of Lemma \ref{Lem: Drift remainder pre estimates} Part \ref{Eqn: Second part estimate} let us take those $(\bq,\omega) \in \tilde{\Lambda}^{\times}_{2n}$ such that every edge is traversed at most twice and write the connected graph $\hat{G} = (V_{\bq,\omega},E_{\bq,\omega})$, then $2n + 2 = 2|E_{\bq,\omega}| - |F_{\bq,\omega}|$. Hence, $|V_{\bq,\omega}| - |F_{\bq,\omega}|/2 = |V_{\bq,\omega}| - |E_{\bq,\omega}| + n + 1 = n + 2 - \beta(\hat{G}) \leq n$, since $\beta(\hat{G}) \geq 2$ for all $(\bq,\omega) \in \tilde{\Lambda}^{\times}_{2n}$. Therefore
\[
\sum_{\omega \in \Lambda^{\times}_{2n}}\sum_{\bq}'\phi_{\omega,\bq}\E[H_{q_0q_1}H_{q_1q_2}H_{\omega}] \ \  = \sum_{(\bq,\omega) \in \tilde{\Lambda}^{\times}_{2n}} \E[H_{\bq,\omega}] = \O(N^n),
\]
which completes the result.
\end{proof}

\subsection{Proof of Proposition \ref{Prop: RITE evolution} Part \ref{Eqn: RITE diffusion} - Diffusion term}\label{Sec: RITE Diffusion}

For the diffusion term we insert the expression (\ref{Eqn: deltaY RITE}) for $\delta Y^{\bq}$ into the definition (\ref{Eqn: RITE expected change multiple variables}) for the expected change of multiple variables to obtain, via Lemma \ref{Lem: Indicator simplification}, 
\begin{eqnarray}
\E[\delta Y_n \delta Y_m | H] & = & \frac{1}{d_N}\sum_{\bq} \Theta_{\bq}(H) \delta Y^{\bq}_n \delta Y^{\bq}_m \nonumber \\
& = &  \frac{1}{(N-2)^{n+m}}\frac{1}{4d_N} \sum_{\omega_1 \in \Lambda_{2n}}\sum_{\omega_2 \in \Lambda_{2m}}H_{\omega_1,\omega_2} \sum'_{\bq}(1 - 3H_{q_0q_1}H_{q_1q_2}) \phi_{\omega_1,\bq}\phi_{\omega_2,\bq} \nonumber \\
& = & \frac{6N}{d_N}[2n^2\delta_{nm} + R_{nm}(H)]. \label{Eqn: Diffusion first equation}
\end{eqnarray}
Let us treat the cases $n = m$ and $n \neq m$ separately. Starting with the former, we define $\Gamma^{\star}_{2n} := \{ (\omega_1, \omega_2) \in \Lambda^{\star}_{2n} \times \Lambda^{\star}_{2n} : \omega_1 \cong \omega_2\}$ and $\Gamma^{\circ}_{2n} = (\Lambda_{2n} \times \Lambda_{2n})\setminus \Gamma^{\star}_{2n}$, exactly as in Section \ref{Sec: ITE Diffusion term} for the ITE. For $(\omega_1,\omega_2) \in \Gamma^{\star}_{2n}$ we have $H_{\omega_1,\omega_2} = 1$ and $\phi_{\omega_1,\bq} = \phi_{\omega_2,\bq}$ for all $\bq$. Therefore,
\begin{multline}\label{Eqn: Rnn expression}
R_{nn}(H) = \frac{1}{(N-2)^{2n}}\frac{1}{24N} \sum_{(\omega_1,\omega_2) \in \Gamma^{\star}_{2n}} \sum'_{\bq}(1 - 3H_{q_0q_1}H_{q_1q_2}) \phi_{\omega_1,\bq}^2 \\
+  \frac{1}{(N-2)^{2n}}\frac{1}{24N} \sum_{(\omega_1,\omega_2) \in \Gamma^{\circ}_{2n}}H_{\omega_1,\omega_2} \sum'_{\bq}(1 - 3H_{q_0q_1}H_{q_1q_2}) \phi_{\omega_1,\bq}\phi_{\omega_2,\bq} - 2n^2.
\end{multline}
Now, $\phi_{\omega,\bq}^2 = \phi_{\omega,\bq}$ and (\ref{Eqn: triangle summation 1}) gives $\sum_{\bq}'\phi_{\omega,\bq} = 12nN + \O(1)$ if $\omega \in \Lambda^{\star}_{2n}$. Furthermore, for every fixed $\omega_1 \in \Lambda^{\star}_{2n}$ there are $4n$ possible walks $\omega_2$ such that $\omega_1 \cong \omega_2$. Therefore, counting the number of ways of labelling the vertices in $\omega_1$ leads to $|\Gamma^{\star}_{2n}| = 4n (N^{2n} + \O(N^{2n-1}))$. Inserting these observations gives
\[
\frac{1}{(N-2)^{2n}}\frac{1}{24N} \sum_{(\omega_1,\omega_2) \in \Gamma^{\star}_{2n}} \sum'_{\bq} \phi_{\omega_1,\bq}^2 = \frac{n}{2(N-2)^{2n}} |\Gamma^{\star}_{2n}|(1 + \O(N^{-1})) = 2n^2 + \O(N^{-1}).
\]
In addition, for a fixed $\bq$ we have $\sum_{\omega \in \Lambda^{\star}_{2n}}\phi_{\omega,\bq} = \O(N^{2n-2})$ since if $\omega$ contains, say, the edge $(q_0,q_1)$ there are $|V_{\omega}|-2 \leq 2n-2$ remaining vertices that can be labelled. Hence, using the regularity of $H$,
\begin{multline}\label{Eqn: Rnn leading contribution}
\frac{1}{(N-2)^{2n}}\frac{1}{24N} \sum_{(\omega_1,\omega_2) \in \Gamma^{\star}_{2n}} \sum'_{\bq} 3H_{q_0q_1}H_{q_1q_2} \phi_{\omega_1,\bq} = \frac{1}{(N-2)^{2n}}\frac{12 n}{24N}  \sum_{\omega \in \Lambda^{\star}_{2n}} \sum'_{\bq}H_{q_0q_1}H_{q_1q_2} \phi_{\omega,\bq} \\
= \O(N^{-3})\sum'_{q_0,q_1}\sum_{q_2 \neq q_0,q_1}H_{q_0q_1}H_{q_1q_2}  = \O(N^{-3})\sum'_{q_0,q_1}H_{q_0q_1}H_{q_1q_0}  = \O(N^{-1}).
\end{multline}
The remainder (\ref{Eqn: Rnn expression}) therefore reduces to
\begin{equation}\label{Eqn: Rnn Simplified expression}
R_{nn}(H) = \frac{1}{(N-2)^{2n}}\frac{1}{24N}\sum_{(\omega_1,\omega_2) \in \Gamma^{\circ}_{2n}}H_{\omega_1,\omega_2} \sum'_{\bq}(1 - 3H_{q_0q_1}H_{q_1q_2}) \phi_{\omega_1,\bq}\phi_{\omega_2,\bq} + \O(N^{-1}).
\end{equation}
Let us define $\tilde{\Gamma}^{\circ}_{2n} := \{(\bq,\omega_1,\omega_2) : q_0 \neq q_1 \neq q_2 \neq q_0, (\omega_1,\omega_2) \in \Gamma^{\circ}_{2n}, \phi_{\omega_1,\bq}\phi_{\omega_2,\bq} = 1\}$ and $H_{\bq,\omega_1,\omega_2} = H_{q_0q_1}H_{q_1q_2}H_{\omega_1}H_{\omega_2}$ as usual. Similarly let $[\bq,\omega_1,\omega_2]$ be the equivalence class of labellings of the vertices of $(\bq,\omega_1,\omega_2)$ and $[\tilde{\Gamma}^{\circ}_{2n}]$ the set of such equivalence classes. Then showing that $\E|R_{nn}(H)| = \O(N^{-1})$ reduces to showing the following is of order $\O(N^{2n})$.
\begin{multline}\label{Eqn: Rnn other contribution}
\E\bigg|\sum_{(\omega_1,\omega_2) \in \Gamma^{\circ}_{2n}}H_{\omega_1,\omega_2}  \sum'_{\bq}(1 - 3H_{q_0q_1}H_{q_1q_2}) \phi_{\omega_1,\bq}\phi_{\omega_2,\bq}  \bigg| \\
\leq
\sum_{[\omega_1,\omega_2] \in [\Gamma^{\circ}_{2n}]} \alpha_{[\omega_1,\omega_2]}  \sqrt{\sum_{(\omega_1,\omega_2),(\omega'_1,\omega'_2) \atop \in [\omega_1,\omega_2]} \E[H_{\omega_1,\omega_2,\omega'_1,\omega'_2}] }
+ \sum_{[\bq,\omega_1,\omega_2] \in [\tilde{\Gamma}^{\circ}_{2n}]} 3\sqrt{ \sum_{(\bq,\omega_1,\omega_2),(\bq',\omega'_1,\omega'_2) \atop \in [\bq,\omega_1,\omega_2] } \E[ H_{\bq,\omega_1,\omega_2,\bq',\omega'_1,\omega'_2}]},
\end{multline}
where $\alpha_{\omega_1,\omega_2} = \sum'_{\bq} \phi_{\omega_1,\bq}\phi_{\omega_2,\bq}$ is the same for all $(\omega_1,\omega_2) \in [\omega_1,\omega_2]$.

We start with the first term in (\ref{Eqn: Rnn other contribution}). For a fixed $(\omega_1,\omega_2)$ $\alpha_{\omega_1,\omega_2} \neq 0$ only if $\omega_1$ and $\omega_2$ are connected, i.e. $|V_{\omega_1}\cap V_{\omega_2}| > 0$. In this case there are two scenario that we must consider. Firstly, $|E_{\omega_1}\cap E_{\omega_2}| = 0$ (no edges are shared) and secondly $|E_{\omega_1}\cap E_{\omega_2}| \geq 1$ (at least one edge is shared). In the first scenario we have $\alpha_{\omega_1,\omega_2} = \O(1)$, since for a fixed $(\omega_1,\omega_2)$ the quantity $\phi_{\omega_1,\bq}\phi_{\omega_2,\bq} \neq 0$ only if $V_{\bq} \subset V_{\omega_1,\omega_2}$ and $|V_{\omega_1,\omega_2}| = \O(1)$. In the second scenario we have $\alpha_{\omega_1,\omega_2} = \O(N)$, since the sharing of an edge between $\omega_1$ and $\omega_2$ only demands that two out of the three vertices $q_0,q_1$ and $q_2$ are fixed, summing over the non-fixed vertex then gives a contribution $\O(N)$.

We therefore want to evaluate the contribution of the quantity inside the square-root of the first term in (\ref{Eqn: Rnn other contribution}) for these two scenarios. To begin, as before, let us define $G = (V_{\omega_1,\omega_2},F_{\omega_1,\omega_2})$ and $\hat{G} = (V_{\omega_1,\omega_2},E_{\omega_1,\omega_2})$. Assuming edges are traversed a maximum of twice in $(\omega_1,\omega_2) \in 
\Gamma^{\circ}_{2n}$ means $4n = 2|E_{\omega_1,\omega_2}| - |F_{\omega_1,\omega_2}|$. For all $(\omega_1,\omega_2)$ such that $\alpha_{\omega_1,\omega_2} > 0$ the subgraph $\hat{G}$ is connected. Also, since $\omega_1$ and $\omega_2$ are non-backtracking cycles, if any of the various connected components of the subgraph $G = (V_{\omega_1,\omega_2},F_{\omega_1,\omega_2})$ have $\beta_i=0$ they must be isolated vertices, as in the Section \ref{Sec: RITE Drift}.

In the first scenario it may be the case that $|F_{\omega_1,\omega_2}| = 4n$ and so the number of isolated vertices satisfies $|V_I| \leq |E_{\omega_1,\omega_2}| - |F_{\omega_1,\omega_2}|$. Hence $|V_I| \leq 4n - |E_{\omega_1,\omega_2}| = 4n - |V_{\omega_1,\omega_2}| - \beta(\hat{G}) + 1$ and so, by Corollary \ref{Cor: RITE vertex edge identity} we have $|V_{\omega_1,\omega_2,\omega'_1,\omega'_2}| - |F_{\omega_1,\omega_2,\omega'_1,\omega'_2}|/2 \leq |V_{\omega_1,\omega_2}| + |V_I| \leq 4n -1$, as $\beta(\hat{G}) \geq 2$ for all $(\omega_1,\omega_2) \in \Gamma^{\circ}_{2n}$. Combining with the contribution of $\alpha_{\omega_1,\omega_2}$ gives a contribution $\O(1)\sqrt{\O(N^{4n-1})} = \O(N^{2n - 1/2})$.

In the second scenario we must have $|F_{\omega_1,\omega_2}| \leq 4n - 2$ as at least one edge must be traversed twice in $(\omega_1,\omega_2)$. This means $|V_I| \leq |E_{\omega_1,\omega_2}| - |F_{\omega_1,\omega_2}| -1 \leq 4n -  |V_{\omega_1,\omega_2}| - \beta(\hat{G})$. Moreover we cannot have $\beta(\hat{G}) = 1$, as this is only possible if $\omega_1$ and $\omega_2$ are single loops and satisfy $\omega_1 \cong \omega_2$, however this would mean $(\omega_1,\omega_2) \notin \Gamma^{\circ}_{2n}$. Thus, by Corollary \ref{Cor: RITE vertex edge identity} we have $|V_{\omega_1,\omega_2,\omega'_1,\omega'_2}| - |F_{\omega_1,\omega_2,\omega'_1,\omega'_2}|/2 \leq |V_{\omega_1,\omega_2}| + |V_I| \leq 4n - 2$. Combining this with $\alpha_{\omega_1,\omega_2}$   leads to a contribution $\O(N)\sqrt{\O(N^{4n-2})} = \O(N^{2n})$ and hence the first term in (\ref{Eqn: Rnn other contribution}) is of order $\O(N^{2n})$.

We now turn our attention to the second term in (\ref{Eqn: Rnn other contribution}). For those $[\bq,\omega_1,\omega_2] \in [\tilde{\Gamma}^{\circ}_{2n}]$ we define the graph $\hat{G} = (V_{\bq,\omega_1,\omega_2},E_{\bq,\omega_1,\omega_2})$. If each of the edges are traversed at most twice then we have $2|E_{\bq,\omega_1,\omega_2}| - |F_{\bq,\omega_1,\omega_2}| = 4n + 2$. Moreover, since there must be a least one edge that is traversed twice we find the number of isolated vertices satisfies $|V_I| \leq |E_{\bq,\omega_1,\omega_2}| - |F_{\bq,\omega_1,\omega_2}| - 1 = 4n + 2 - |V_{\bq,\omega_1,\omega_2}| - \beta(\hat{G})$. Thus, using that $\beta(\hat{G}) \geq 2$ and Corollary \ref{Cor: RITE vertex edge identity} we find $|V_{\bq,\omega_1,\omega_2,\bq',\omega'_1,\omega'_2}| - |F_{\bq,\omega_1,\omega_2,\bq',\omega'_1,\omega'_2}|/2 \leq 4n$ for all $(\bq,\omega_1,\omega_2) \sim (\bq',\omega'_1,\omega'_2) \in \tilde{\Gamma}^{\circ}_{2n}$. The contribution of the second term is therefore $\O(N^{2n})$, as desired.

\vspace{10pt}

It thus remains to evaluate $\E|R_{nm}(H)|$ for $n \neq m$. If we define $\tilde{\Gamma}_{2n,2m} = \{(\bq,\omega_1,\omega_2) : \omega_1 \in \Lambda_{2n}, \omega_2 \in \Lambda_{2m}, \phi_{\omega_1,\bq}\phi_{\omega_2,\bq} = 1\}$ then from (\ref{Eqn: Diffusion first equation}) we have
\begin{multline}
\E|R_{nm}(H)| \leq \frac{1}{(N-2)^{n+m}}\frac{1}{24 N} \bigg\{ \sum_{[\omega_1,\omega_2] \in [\Lambda_{2n} \times \Lambda_{2m}]} \alpha_{[\omega_1,\omega_2]}
\E\bigg| \sum_{(\omega_1,\omega_2)\in [\omega_1,\omega_2]} H_{\omega_1,\omega_2}\bigg| \\
+ \sum_{[\omega_1,\omega_2] \in [\tilde{\Gamma}_{2n,2m}]} 
\E\bigg| \sum_{(\bq,\omega_1,\omega_2) \in [\bq,\omega_1,\omega_2]} H_{\bq,\omega_1,\omega_2}\bigg| \bigg\}.
\end{multline}
For the first term we can use the same arguments as for $n=m$. In particular, if $\alpha_{\omega_1,\omega_2} >0$ then $\omega_1$ and $\omega_2$ are connected. If they do not share an edge (i.e. $|E_{\omega_1} \cap E_{\omega_2}| = 0$) then it is possible the number of non-free edges $|E_{\omega} \setminus F_{\omega}| = 0$. Hence $|V_I| \leq |E_{\omega}| - |F_{\omega}|$, which in turn implies (assuming edges are not traversed more than twice by $(\omega_1,\omega_2)$) $|V_{\omega_1,\omega_2}| + |V_I| \leq 2n + 2m + 1 - \beta(\hat{G}) \leq 2n + 2m -1$, as $\beta(\hat{G}) \geq 2$. Moreover, $\alpha_{\omega_1,\omega_2} = \O(1)$ if $|E_{\omega_1} \cap E_{\omega_2}| = 0$ so the contribution is of order $\O(1)\sqrt{\O(N^{2n + 2m-1})} = \O(N^{n+m - 1/2})$. Alternatively, if $\omega_1$ and $\omega_2$ share at least one edge (i.e. $|E_{\omega_1} \cap E_{\omega_2}| =>0$) then we must have $|V_I| \leq |E_{\omega}| - |F_{\omega}| -1 $, which in turn implies $|V_{\omega_1,\omega_2}| + |V_I| \leq 2n + 2m - \beta(\hat{G}) \leq 2n + 2m - 2$, as $\beta(\hat{G}) \geq 2$. Therefore, since $\alpha_{\omega_1,\omega_2} = \O(N)$ in this case, we attain a contribution of $\O(N)\sqrt{\O(N^{2n + 2m-2})} = \O(N^{n+m})$ for the first term.

Similarly for the second term we write $\hat{G} = (V_{\bq,\omega_1,\omega_2},E_{\bq,\omega_1,\omega_2})$ and, assuming the edges are traversed at most twice, we get $|V_{\bq,\omega_1,\omega_2}| + |V_{I}| \leq 2n + 2m + 2 - \beta(\hat{G})$. Thus Corollary \ref{Cor: RITE vertex edge identity} means we have $|V_{\bq,\omega_1,\omega_2,\bq',\omega'_1,\omega'_2}| - |F_{\bq,\omega_1,\omega_2,\bq',\omega'_1,\omega'_2}|/2 \leq 2n + 2m$, as $\beta(\hat{G}) \geq 2$. Therefore the contribution of the second term is of order $\O(N^{n+m})$ also.

Finally noting that there is a factor of order $O(N^{-n-m-1})$ means that $\E|R_{nm}(H)| = \O(N^{-1})$.

\subsection{Proof of Proposition \ref{Prop: RITE evolution} Part \ref{Eqn: RITE remainder} - Remainder term}\label{Sec: RITE Remainder}
For the remainder term we again insert the expression (\ref{Eqn: deltaY RITE}) for $\delta Y^{\bq}$ into the definition (\ref{Eqn: RITE expected change multiple variables}) so that, using Lemma \ref{Lem: Indicator simplification}, we obtain
\begin{align*}
\E[|\delta Y_n \delta Y_m \delta Y_l | |H] & = \frac{1}{d_N}\sum'_{\bq}\Theta_{\bq}(H) \left| \delta Y^{\bq}_n \delta Y^{\bq}_m \delta Y^{\bq}_l \right| \\
& \leq  \frac{1}{(N-2)^{n+m+l}}\frac{1}{4d_N}\sum'_{\bq} |(1 - 3H_{q_0q_1}H_{q_1q_2})| \bigg|  \sum_{(\omega_1 \omega_2 \omega_3) \atop \in \Lambda_{2n} \times \Lambda_{2m} \times \Lambda_{2l}}H_{\omega_1,\omega_2,\omega_3}\phi_{\omega_1,\bq}\phi_{\omega_2,\bq}\phi_{\omega_3,\bq}\bigg|. 
\end{align*}
Therefore, given that $|(1 - 3H_{q_0q_1}H_{q_1q_2})| \leq 4$ we find
\begin{equation}\label{Eqn: RITE Remainder evaluation}
\E|R_{nml}(H)| = \frac{d_N}{6 N}\E[|\delta Y_n \delta Y_m \delta Y_l | |H] 
\leq \frac{1}{(N-2)^{n+m+l}}\frac{1}{6N}\sum'_{\bq} \E\left|\sum_{(\omega_1,\omega_2,\omega_3) \in \Gamma^{\bq}_{2n,2m,2l}} H_{\omega_1,\omega_2,\omega_3}\right|,
\end{equation}
where $\Gamma^{\bq}_{2n,2m,2l} = \{(\omega_1,\omega_2,\omega_3) \in \Lambda_{2n} \times \Lambda_{2m} \times \Lambda_{2l} : \phi_{\omega_1,\bq}\phi_{\omega_2,\bq}\phi_{\omega_3,\bq} = 1\}$. Using the standard inequality for the expectation means we must compute the quantity
\begin{equation}\label{Eqn: RITE Remainder evaluation extra}
\E\bigg|\sum_{(\omega_1,\omega_2,\omega_3) \in \Gamma^{\bq}_{2n,2m,2l}} H_{\omega_1,\omega_2,\omega_3}\bigg| \leq \sum_{[\omega_1,\omega_2,\omega_3]  \in [\Gamma^{\bq}_{2n,2m,2l}]} \sqrt{\sum_{(\omega_1,\omega_2,\omega_3),(\omega'_1,\omega'_2,\omega'_3) \atop \in [\omega_1,\omega_2,\omega_3]} \E[H_{\omega_1,\omega_2,\omega_3,\omega'_1,\omega'_2,\omega'_3}]}.
\end{equation}
The condition $\phi_{\omega_1,\bq}\phi_{\omega_2,\bq}\phi_{\omega_3,\bq}=1$ imposes the restriction that $\omega_1, \omega_2$ and $\omega_3$ must share an odd number of edges with $\bq$. Let us restrict ourselves to those $(\omega_1,\omega_2,\omega_3) \in \Lambda^{\dagger}_{2n} \times \Lambda^{\dagger}_{2m} \times \Lambda^{\dagger}_{2l}$ (i.e. $|V_{\omega_1}| = 2n$ etc.) as this maximises the number of vertices and therefore gives the main contribution to (\ref{Eqn: RITE Remainder evaluation}). Note that for all $\omega \in \Lambda^{\dagger}_{2n}$, $\phi_{\bq,\omega} = 1$ if and only if $|E_{\omega} \cap \{(q_0,q_1),(q_1,q_2),(q_2,q_0)\}| = 1$. There are (up to the relabelling of vertices) three scenarios 
\begin{enumerate}
\item [(i)] $|E_{\omega_1} \cap \{(q_0,q_1)\}| = |E_{\omega_2} \cap \{(q_1,q_2)\}| = |E_{\omega_3} \cap \{(q_2,q_0)\}| = 1$
\item [(ii)] $|E_{\omega_1} \cap \{(q_0,q_1)\}| = |E_{\omega_2} \cap \{(q_0,q_1)\}| = |E_{\omega_3} \cap \{(q_1,q_2)\}| = 1$
\item [(iii)] $|E_{\omega_1} \cap \{(q_0,q_1)\}| = |E_{\omega_2} \cap \{(q_0,q_1)\}| = |E_{\omega_3} \cap \{(q_0,q_1)\}| = 1$
\end{enumerate}
In the first scenario we assume that each edge is traversed at most twice by $(\omega_1,\omega_2,\omega_3)$. Then, following the same arguments as Sections \ref{Sec: RITE Drift} and \ref{Sec: RITE Diffusion}, the number of isolated vertices $|V_I| \leq |E_{\omega_1,\omega_2,\omega_3}| - |F_{\omega_1,\omega_2,\omega_3}| \leq 2n + 2m + 2l - |V_{\omega_1,\omega_2,\omega_3}| + 1 - \beta(\hat{G})$, where $\hat{G} = (V_{\omega_1,\omega_2,\omega_3},E_{\omega_1,\omega_2,\omega_3})$. By construction we must have $\beta(\hat{G}) \geq 4$ and thus $|V_{(\omega_1,\omega_2,\omega_3)}| + |V_I| \leq 2m + 2n + 2l - 3$.

In the second scenario we also assume that each edge is traversed at most twice by $(\omega_1,\omega_2,\omega_3)$. However at least one edge (given by $(q_0,q_1)$) must be traversed twice, which means the number of isolated vertices satisfies $|E_{\omega_1,\omega_2,\omega_3}| - |F_{\omega_1,\omega_2,\omega_3}| - 1$ (it always at least one less than  the number of non-free edges). Therefore we find $|V_{\omega_1,\omega_2,\omega_3}| + |V_I| \leq 2n + 2m + 2l - \beta(\hat{G}) = 2n + 2m + 2l - 3$, since we must have $\beta(\hat{G}) \geq 3$, with equality occurring when $\omega_1 \cong \omega_2$ (which only happens when $n = m$).

Using Corollary \ref{Cor: RITE vertex edge identity} we determine that in Scenarios (i) and (ii)
\[
|V_{\omega_1,\omega_2,\omega_3,\omega'_1,\omega'_2,\omega'_3}| - \frac{|F_{\omega_1,\omega_2,\omega_3,\omega'_1,\omega'_2,\omega'_3}|}{2} \leq 2n + 2m + 2l - 3.
\]
However all three vertices $q_0,q_1$ and $q_2$ are fixed for a particular $\Gamma^{\bq}_{2n,2m,2l}$ so the contribution to (\ref{Eqn: RITE Remainder evaluation extra}) is of order $\sqrt{\O(N^{2n + 2m + 2l - 3 - 3})} = \O(N^{n+m+l - 3})$.

In the final scenario the edge $(q_0,q_1)$ is traversed 3 times. We assume that all others are traversed at most twice, which implies that $2n + 2m + 2l = 2|E_{\omega_1,\omega_2,\omega_3}| - |F_{\omega_1,\omega_2,\omega_3}| + 2$. Now, either $\omega_i \cong \omega_j$ for some $i\neq j = 1,2,3$ or not. Let us suppose the former case arises (we can say $\omega_1 \cong \omega_2$ for instance) then, because at least one edge is traversed twice, the number of isolated vertices is given by $|V_I| = |E_{\omega_1,\omega_2,\omega_3}| - |F_{\omega_1,\omega_2,\omega_3}| - 1$, which means $|V_{\omega_1,\omega_2,\omega_3}| + |V_I| = 2n + 2m + 2l - \beta(\hat{G}) - 2 \leq 2n + 2m + 2l - 4$, since $\beta(\hat{G}) \geq 2$ in this case. 

Alternatively, if $\omega_i \not\cong \omega_j$ for any $i\neq j$ then $|V_I| = |E_{\omega_1,\omega_2,\omega_3}| - |F_{\omega_1,\omega_2,\omega_3}|$, but we must have $\beta(\hat{G}) \geq 3$, which means again, $|V_{\omega_1,\omega_2,\omega_3}| + |V_I| \leq 2n + 2m + 2l - \beta(\hat{G}) - 1 \leq 2n + 2m + 2l - 4$. Corollary \ref{Cor: RITE vertex edge identity} therefore implies that for Scenario (iii) we have 
\[
|V_{\omega_1,\omega_2,\omega_3,\omega'_1,\omega'_2,\omega'_3}| - \frac{|F_{\omega_1,\omega_2,\omega_3,\omega'_1,\omega'_2,\omega'_3}|}{2} \leq 2n + 2m + 2l - 4.
\]
However this time there are only two vertices of $\bq$ contained in $V_{\omega_1,\omega_2,\omega_3}$, which means the contribution to (\ref{Eqn: RITE Remainder evaluation extra}) is of order $\sqrt{\O(N^{2n + 2m + 2l - 4 - 2})} = \O(N^{n+m+l - 3})$ once again.

Returning to (\ref{Eqn: RITE Remainder evaluation}) and noting that the sum over $\bq$ gives a contribution of $\O(N^3)$ means that $\E|R_{nml}(H)| = \O(N^{-1})$, as desired.

\section{Conclusions}\label{Sec: Conclusion}
We have used a combination of appropriate random walks and Stein's method to provide rates of convergence for the traces of random Bernoulli ensembles derived from both tournaments and regular tournaments. Specifically we have shown that under this random walk the traces, in a basis of Chebyshev polynomials, behave like independent Ornstein-Uhlenbeck processes in the limit of large matrix size. Subsequently, this allows to use the results of Chatterjee \& Meckes \cite{Chatterjee-2008}, Reinert \& R\"{o}llin \cite{Reinert-2009} and Meckes \cite{Meckes-2009}, regarding the multivariate version of the exchangeable pairs mechanism for Stein's method, in order to obtain rates of convergence to an appropriate Gaussian distribution. In particular, we are able to obtain these results using combinatorial methods, closely related to previous calculations for showing distributional convergence, but without explicit rates (see e.g. \cite{Schenker-2007,Sodin-2017}). Moreover, this approach only requires estimates involving third order moments to show distributional convergence.

We would like to finish with a couple of comments. Firstly, we note that in the bound for the distributional distance (\ref{Eqn: RITE distributional distance}) of the RITE in Theorem \ref{Thm: RITE theorem}, the first term is of order $\O(N^{-1/2})$. This comes from a single set of walks, arising due to the regularity of the matrix $H$ (see the proof of Lemma \ref{Lem: Drift remainder pre estimates} Part (b)). It is not clear whether this can be improved to $\O(N^{-1})$ in order to match the corresponding result in Theorem \ref{Thm: ITE theorem} for the ITE. Secondly, we believe the results could be easily applied to other types of matrix ensembles such as Wigner matrices, or tournaments with different score sequences. For Wigner matrices the random walk would be very similar - one may choose a matrix element at random and then resample from the appropriate distribution. However Lemma \ref{Lem: Non-backtracking} is not immediately applicable and would therefore have to be amended. Although results in this direction have already been achieved \cite{Feldheim-2010}. For tournaments with different score sequences similar random walks to the RITE have already been analysed \cite{Kannan-1997} and the number of such tournaments have been asymptotically estimated \cite{McKay-1996}, expanding on the technques developed by McKay for regular tournaments \cite{McKay-1990}, which suggests a result akin to Lemma \ref{Lem: Regular tournament expectation} would also be possible.

\subsection*{Acknowledgements} The authors are grateful to Prof. Sasha Sodin for numerous discussions and useful comments regarding this work. CHJ would also like to acknowledge the Leverhulme Trust for financial support (ECF-2014-448)

\appendix

\section*{Appendix}

\section{Stein solution}\label{Sec: Stein solution}
 
Since we detail a slightly different (and more specific) version of Theorem \ref{Thm: Stein theorem} to that of Theorem 2.1 in \cite{Reinert-2009} and Theorem 3 in \cite{Meckes-2009} we have decided to include a short proof for the aid of the reader. In particular, \cite{Reinert-2009,Meckes-2009} allow for a multivariate Gaussian distribution with general covariance matrix $\Sigma$, which we have decided to specify to our situation for clarity. Moreover, the bounds in \cite{Reinert-2009} are also in terms of the derivatives $\| \nabla^{j} \phi \|$ but of an order one more than presented here. This realisation that the order can be reduced by one through integration by parts (see Lemma \ref{Lem: Function bounds}) is presented in \cite{Meckes-2009} but this 
is done with a more complicated type of function bound and so we keep with derivatives of the form $\| \nabla^{j} \phi \|$ for simplicity. This does not provide any meaningful effect on our final result.

\begin{proposition}[Stein solution]\label{Prop: Stein sol}Let $\A$ be the operator given in (\ref{Eqn: Stein operator}). Then the solution to Stein's equation (\ref{Eqn: Stein equation}) is given by
\[
f(X) = \int_0^{\infty} dt \int dZ' \ P(X \to Z' ; t) \phi(Z') = \int_0^{\infty} dt \int d\mu(Z) \ \phi(\tilde{X}(X,Z;t)) \ ,
\]
where $P(X \to Z';t) = \prod_{n=2}^k P_n(X_n \to Z'_n;t)$,
\begin{equation}\label{Eqn: OU transition probability}
P_n(X_n \to Z'_n;t) := \frac{1}{\sqrt{2 \pi n} \chi_n(t)}\exp\left(-\frac{(X_ne^{-nt} - Z'_n)^2}{2n\chi_n(t)^2}\right) \ , \hspace{10pt} \chi_n(t)^2 = 1 - e^{-2nt},
\end{equation}
describes the evolution in the corresponding one-dimensional OU process for a fixed initial position $X_n$.

After a simple change of variables one obtains the second equality where $Z = (Z_2,\ldots,Z_k)$, $Z_n \sim N(0,n)$ are independent Gaussian random variables with $\mu(Z)$ the associated measure and $\tilde{X} = (\tilde{X}_2,\ldots,\tilde{X}_k)$ with $\tilde{X}_n(X,Z;t) = X_n e^{-nt} + \sqrt{1 - e^{-2nt}}Z_n$.
\end{proposition}
\begin{proof}
Let us write $\A:= \sum_{n=2}^k n\L_n $ with
\[
\L_n:= n\frac{\partial^2}{\partial Z'^2} - Z'\frac{\partial }{\partial Z'},
\]
then the solution to the backward Fokker-Planck equation $\partial_t P (Z';t) = \A P (Z';t)$ is well-known (see e.g. \cite{Wang-1945,Risken-1989}) and given by (\ref{Eqn: OU transition probability}). Therefore
\begin{multline}
\A f(X) =  \int_0^{\infty} dt \int dZ \partial_t P(X \to Z';t) \phi(Z') =  \int dZ' [P(X,Z';\infty) - P(X,Z';0)] \phi(Z') \\ 
=  \int dZ [P(Z) - \delta(X-Z)] \phi(Z)  = \E[\phi(Z)] - \phi(X).
\end{multline}
\end{proof}

\begin{remark}The operator $\A^*$, as mentioned in Section \ref{Sec: Random walks}, is the generator for the corresponding \emph{forward} Fokker-Planck equation, given by 
\begin{equation}\label{Eqn: FP eqn}
\frac{\partial P(X,t)}{\partial t} = [\A^*P](X;t) : =  \sum_{n=2}^k n \left[\frac{\partial(X_n P(X;t))}{\partial X_n} + n \frac{\partial^2 P(X;t)}{\partial X_n^2}\right].
\end{equation}
\end{remark}

\begin{lemma}\label{Lem: Function bounds}Let $f$ be connected to $\phi \in C^3(\R^{k-2})$ as in Proposition \ref{Prop: Stein sol}. Then
\begin{equation}\label{Eqn: Function bounds}
\left\| \nabla^j f \right\| \leq \frac{1}{\sqrt{\pi}}\frac{2^{j-3}\Gamma(\frac{k}{2})^2}{(k-1)!}\left\| \nabla^{j-1} \phi\right\|,
\end{equation}
with $\|\nabla ^j f \|$ and $\|\nabla^j \phi \|$ defined in (\ref{Eqn: Sup deriv norm}).
\end{lemma}

\begin{proof}
We have, writing $d\mu(Z) = dZP(Z)$ and changing variables of the derivatives 
\begin{multline}
\frac{\partial^j f(X)}{\partial X_{n_1}\ldots \partial X_{n_j}}  =  \int_0^{\infty} dt \int dZ P(Z) \frac{\partial^j \phi(\tilde{X})}{\partial X_{n_1}\ldots \partial X_{n_j}} \\
 =  -\int_0^{\infty} dt \ \frac{e^{-(n_1 + \ldots +  n_j)t}}{\sqrt{1 - e^{-2n_j t}}}  \int dZ P(Z) \frac{\partial^j \phi(\tilde{X})}{\partial \tilde{X}_{n_1}\ldots \partial \tilde{X}_{n_{j-1}}\partial Z_{n_j}} \ ,
\end{multline}
where $n_i= 2,\ldots,k$. Integration by parts may therefore be performed on the $Z_{n_j}$ variable
\begin{multline}
\frac{\partial^j f(X)}{\partial X_{n_1}\ldots \partial X_{n_j}} =  \int_0^{\infty}  dt \frac{e^{-(n_1 + \ldots +  n_j)t}}{\sqrt{1 - e^{-2n_j t}}}  \left\{ \int dZ \frac{\partial P(Z)}{\partial Z_{n_j}} \frac{\partial^{j-1} \phi(\tilde{X})}{\partial \tilde{X}_{n_1}\ldots \partial \tilde{X}_{n_j}}  - \left[ P(Z) \frac{\partial^{j-1} \phi(\tilde{X})}{\partial \tilde{X}_{n_1}\ldots \partial \tilde{X}_{n_{j-1}}}\right]_{-\infty}^{\infty}\right\}  \\
 =  -\int_0^{\infty} dt  \frac{e^{-(n_1 + \ldots + n_j)t}}{\sqrt{1 - e^{-2n_j t}}}  \int dZ P(Z) \frac{Z_{n_j}}{n_j} \frac{\partial^{j-1} \phi(\tilde{X})}{\partial \tilde{X}_{n_1}\ldots \partial \tilde{X}_{n_{j-1}}} .
\end{multline}
Thus, using $\E[|Z_{n_j}|] = \sqrt{\frac{2n_j}{\pi}}$ for $Z_{n_j} \sim N(0,n_j)$, gives
\[
\left| \frac{\partial^j f(X)}{\partial X_{n_1}\ldots \partial X_{n_j}}\right| \leq    \sup_{\tilde{X}} \left| \frac{\partial^{j-1} \phi(\tilde{X})}{\partial \tilde{X}_{n_1}\ldots \partial \tilde{X}_{n_{j-1}}}\right| \sqrt{\frac{2}{\pi n_j}} \int_0^{\infty} dt  \frac{e^{-(n_1 + \ldots + n_j)t}}{\sqrt{1 - e^{-2n_j t}}} .
\]
Finally, since $n_i \geq 2$ we have $e^{- (n_1 + \ldots + n_j) t} \leq e^{-2jt}$, $(1-e^{-2n_j t})^{-\frac{1}{2}} \leq (1-e^{-4t})^{-\frac{1}{2}}$ and
\[
 \int_0^{\infty} dt  \frac{e^{-2j t}}{\sqrt{1 - e^{-4t}}} = \frac{1}{2} \frac{2^{j-2} \Gamma(j/2)^2}{(j-1)!},
\]
which leads directly to (\ref{Eqn: Function bounds}).
\end{proof}

\begin{proof}[Proof of Theorem \ref{Thm: Stein theorem}]Let $f$ be connected to $\phi$ via the Stein equation (\ref{Eqn: Stein equation}). Since $(M,M')$ are an exchangeable pair so are the random variables $Y':= Y(M')$ and $Y := Y(M)$, hence $E[\delta f] = \E[f(Y')] - \E[f(Y)] = 0$. Therefore, expanding $f(Y')$ in a Taylor series about $Y$ and substituting for the expressions (\ref{Eqn: Drift term}) and (\ref{Eqn: Diffusion term}) we get
\begin{align*}
 0 & =  \frac{1}{C_N}(\E[f(Y')] - \E[f(Y)]) \nonumber \\
 & =  \frac{1}{C_N}\E\left[\sum_{n=2}^k \E[\delta Y_n|M]\frac{\partial f}{\partial Y_n} + \frac{1}{2} \sum_{n,m=2}^k \E[\delta Y_n \delta Y_m|M]\frac{\partial^2 f}{\partial Y_n\partial Y_m} + \E[S_f(M,M')|M] \right] \nonumber \\
 & =  \E[\A f(Y(M))] + \E\left[\sum_{n=2}^k R_n(M)\frac{\partial f}{\partial Y_n} +\frac{1}{2} \sum_{n,m=2}^k R_{nm}(M)\frac{\partial^2 f}{\partial Y_n\partial Y_m} + \frac{1}{C_N}\E[S_f(M,M')|M] \right], \nonumber 
\end{align*}
where $S_f(M,M')$ is the integral form of the remainder obtained in Taylor's theorem
\[
S_f(M,M')  = \frac{1}{3!}\sum_{n,m,l=2}^k\delta Y_n \delta Y_m \delta Y_l \int_0^1 dv (1-v)^2\frac{\partial^3f((1-v)Y + vY')}{\partial Y_n \partial Y_m \partial Y_l}.
\]
Using $\int_0^1 dv (1-v)^2 = \frac{1}{3}$ means $|S_f(M,M')| \leq \frac{1}{3!} \frac{1}{3}\|\nabla^3 f\| \sum_{n,m,l=2}^k |\delta Y_n \delta Y_m \delta Y_l |$
and so a direct substitution of Stein's equation (\ref{Eqn: Stein equation}) yields
\begin{align}
 |\E[\phi(Y)] - \E[\phi(Z)]| & \leq  \|\nabla f\| \sum_{n=2}^k \E|R_n(M)|  +\frac{1}{2} \|\nabla^2 f\| \sum_{n,m=2}^k \E|R_{nm}(M)|  + \frac{1}{18} \|\nabla^3 f\| \sum_{n,m,l=2}^k \E|R_{nml}(M)| \nonumber \\
& =  \Rc^{(1)}\|\nabla f\|  +\frac{1}{2} \Rc^{(1)} \|\nabla^2 f\| + \frac{1}{18} \Rc^{(1)} \|\nabla^3 f\|. \nonumber 
\end{align}
Finally, using Lemma \ref{Lem: Function bounds} we have $\|\nabla^j f\| \leq r_j\|\nabla^{j-1} \phi\|$ with explicit values for the $r_j$.
\end{proof}

\section{Expectations in the RITE}\label{App: Expectation in RITE}

\begin{proof}[Proof of Lemma \ref{Lem: Regular tournament expectation}]In order to prove the lemma we use the ideas of McKay \cite{McKay-1990}, who was originally interested in establishing the asymptotic number of regular tournaments. This was achieved via what he describes as a saddle-point argument, which we adapt here for our current purposes. The main idea is to rewrite the expectation $\E[H_E]$ in terms of a trigonometric integral (see Equation (\ref{Eqn: expectation expression})), with $N$ angles $\theta_p$ corresponding to each of the $N$ rows in the matrix $H$. Crucially the integrand depends only on the differences $\theta_p - \theta_q$ of these angles and is maximised when all angles are equal. Therefore we show the main contribution comes from the region where $\theta_p \approx \theta_q$ for all $p,q$ and the remaining regions are negligible in the limit of large $N$.

To construct the appropriate integral expression we begin with the following characteristic function
\[
\chi_{\RT_N}(H) = \left\{\begin{array}{ll}
0 & H \notin \RT_N \\
1 & H \in \RT_N . \end{array}\right.
\]
An analytical expression for $\chi_{\RT_N}(H)$ may be achieved via the Kronecker delta function. If we let $S_p = -\sum_q iH_{pq}$ be the row sums then our matrix $H$ belongs to $\RT_N$ only if $S_p = 0$ for all $p$. Therefore
\begin{equation}\label{Eqn: RT char fnc}
\chi_{\RT_N}(H)  =  \prod_p \delta_{S_p,0} = \prod_p  \frac{1}{2\pi} \int_0^{2 \pi} d \theta_p \exp\left(i S_p \theta_p\right)
 = \frac{1}{(2\pi)^N} \int_{0}^{2\pi} d^N \theta \prod_{p < q}  \exp\left(H_{pq}( \theta_p - \theta_q) \right),
\end{equation}
where we have used that $H_{pq} = -H_{qp}$. We notice in the expressions above that, since $S_p$ is always even, the integrand is invariant under the shift $\theta_p \mapsto \theta_p + \pi$ for any $p$, and so
\begin{equation}
\chi_{\RT_N}(H)  =  \frac{1}{\pi^N} \int_{-\frac{\pi}{2}}^{\frac{\pi}{2}} d^N \theta \prod_{p < q}  \exp\left(H_{pq}( \theta_p - \theta_q) \right).
\end{equation}
Summing over all possible matrices $H \in \TT_N$ and weighting by this characteristic function leads to the following integral expression for the number of regular tournaments and evaluated by McKay \cite{McKay-1990}
\begin{multline}\label{Eqn: RT integral}
|\RT_N| = \sum_{H \in \TT_N} \chi_{\RT_N}(H) = \frac{2^{N(N-1)/2}}{\pi^N} \int_{-\frac{\pi}{2}}^{\frac{\pi}{2}} d^N \theta \prod_{ p < q }  \cos( \theta_p - \theta_q) \\
=  \frac{2^{N(N-1)/2}}{\pi^{N-1}}\left(\frac{N}{e}\right)^{\frac{1}{2}}\left(\frac{2 \pi}{N}\right)^{\frac{N-1}{2}}(1 + \O(N^{-\frac{1}{2}+\epsilon})).
\end{multline}
Using the same approach we can evaluate the expectation in Lemma \ref{Lem: Regular tournament expectation}. Using the characteristic function (\ref{Eqn: RT char fnc}) the expectation (\ref{Eqn: RT expectation}) is therefore
\begin{equation}\label{Eqn: expectation expression}
\E_{\RT_N}[H_{E}] = \frac{1}{|\RT_N|} \sum_{H \in \TT_N}H_{E}\chi_{\RT_N}(H) = \frac{2^{N(N-1)/2}i^k}{\pi^N |\RT_N|} I,
\end{equation}
where
\[
I = \int_{-\frac{\pi}{2}}^{\frac{\pi}{2}} d^N \theta  \prod_{(p,q) \in E}  \sin(\theta_p - \theta_q)  \prod_{(p,q) \in E^c}\cos( \theta_p - \theta_q)
\]
and $E^c = \{(p,q) : 1 \leq p < q \leq N \}\setminus E$.

To evaluate the integral $I$ we split the integration range into those parts which are dominant and subdominant. To this end let us define the following quantities
\begin{itemize}
\item $A_s = [(s-4)\pi/8,(s-3)\pi/8]$, so in particular $[-\frac{\pi}{2},\frac{\pi}{2}] = \bigcup_{s=0}^{7} A_s$.
\item We therefore write $\bs = (s_1,\ldots,s_{N})$, for $s_p =0,1,\ldots,7$ to signify that the $N$-tuple of angles is in the specific region $(\theta_1,\theta_2,\ldots,\theta_N) \in V(\bs) = A_{s_1} \times A_{s_2} \times \ldots \times A_{s_N}$.
\item $n_j = n_j(\bs) = \#\{p :s_p = j\}$. This counts the number of angles $\theta_p$ in the segment $A_j$. 
\item $D^{(1)} = \{\bs : n_j + n_{j+1(8)}+ n_{j+2(8)} + n_{j+3(8)} = N \ \text{for some} \ j=0,\ldots,7\}$ (where $i+1(8)$ refers to $i+1$ modulo 8 etc.). Thus if $\bs \in D^{(1)}$ this means all angles $\theta_p$ are contained in the region $A_j \cup A_{j+1(8)} \cup A_{j+2(8)} \cup A_{j+3(8)}$, for some $j$.
\item $D^{(2)} = \{ \bs \in \{0,\ldots,7\}^{N}\} \setminus D^{(1)}$ denotes all other possible placements of the angles $(\theta_1,\theta_2,\ldots,\theta_N)$.
\end{itemize}

Thus, $I = J^{(1)} + J^{(2)}$, where 
\begin{equation}\label{Eqn: J integral}
J^{(t)} = \sum_{\bs \in \D^{(t)}} \int_{V(\bs)}  d^N \theta  \prod_{(p,q) \in E}  \sin(\theta_p - \theta_q)  \prod_{(p,q) \in E^c}\cos( \theta_p - \theta_q).
\end{equation}
We will show subsequently that 
\begin{equation}\label{Eqn: J1 result}
|J^{(1)}| \leq  \sqrt{N} \left(\frac{2\pi}{N}\right)^{\frac{N-1}{2}} \O\left(N^{-\frac{k}{2}},\right)
\end{equation}
where $k = |E|$ is the number of edges in $E$, and $|J^{(2)}|$ is negligible in comparison in the large $N$ limit. Hence, 
inserting the expression for $|\RT_N|$, gives $\E_{\RT_N}[H_{E}] = \O(N^{-k/2})$, as desired.
\vspace{15pt}

We begin by showing the result (\ref{Eqn: J1 result}) for $J^{(1)}$ which provides the leading contribution. From the form of $D^{(1)}$ we see that all angles are contained in a range $[-\frac{\pi}{2},\frac{\pi}{2}]$ up to translation\footnote{This holds true even if the sets are disconnected, e.g. if all angles are contained in $A_0 \cup A_1 \cup A_6 \cup A_7$ then we can first make a translation of $\theta_p \mapsto \theta_p - \pi$ to all angles in $A_6 \cup A_7$ which does not change the value of the integral, so all angles are contained in $[-\pi,\pi/2]$.}. The sets $D^{(1)}_i: =  \{\bs : n_i + n_{i+1(8)} + n_{i+2(8)} + n_{i+3(8)} = N \}$ are not necessarily disjoint for $i \neq j$ so $D^{(1)} \subset \bigcup_{i=0}^7 D^{(1)}_i$. But the sum in (\ref{Eqn: J integral}) is the same when restricted to any of the $D^{(1)}_i$. Thus
\begin{multline}
|J^{(1)}| \leq  \sum_{i=0}^7  \sum_{\bs \in \D^{(1)}_i} \int_{V(\bs)} d^N \theta \prod_{(p,q) \in E}  |\sin(\theta_p - \theta_q)|  \prod_{(p,q) \in E^c}|\cos( \theta_p - \theta_q)|  \\
= 8 \int_{-\frac{\pi}{4}}^{\frac{\pi}{4}} d^N \theta \prod_{(p,q) \in E}  |\sin(\theta_p - \theta_q)|  \prod_{(p,q) \in E^c}|\cos( \theta_p - \theta_q)| .
\end{multline}
We are now in a position to use the following bounds
\begin{equation}\label{Eqn: sin and cos bounds}
|\sin(x)| \leq |x| \exp\left(-\frac{1}{2}x^2\right) \qquad \text{and} \qquad
|\cos(x)| \leq \exp\left(-\frac{1}{2}x^2\right),
\end{equation}
which are valid for  $|x| \leq \frac{\pi}{2}$. Inserting these, employing the transition $\theta_p \mapsto \theta_p - \theta_N$ for all $p=1,\ldots,N-1$, integrating over the redundant $\theta_N$ and extending the integration range to the whole real line leads to
\begin{multline}
|J^{(1)}| \leq 8 \int_{-\frac{\pi}{4}}^{\frac{\pi}{4}} d^N \theta \prod_{(p,q) \in E} |\theta_p - \theta_q | \exp\left(-\frac{1}{2}\sum_{1\leq p<q \leq N} (\theta_p - \theta_q)^2 \right) \\
\leq \frac{8\pi}{2} \int_{-\infty}^{\infty} d^{N-1} \theta \prod_{(p,q) \in E} |\theta_p - \theta_q | e^{-\frac{1}{2}\theta^T\Sigma_{N-1}^{-1}\theta},
\end{multline}
where $\theta = (\theta_1,\ldots,\theta_{N-1})$.  The covariance matrix is therefore
\[
\Sigma_{N-1}^{-1} = N\mathbf{I}_{N-1} - \mathbf{E}_{N-1}.
\]
Here $\mathbf{I}_{r}$ denotes the $r \times r$ identity matrix and $\mathbf{E}_{r}$ the $r \times r$ matrix in which every element is 1. The inverse can be easily verified to be
\begin{equation}\label{Eqn: Covariance matrix}
\Sigma_{N-1} = \frac{1}{N}(\mathbf{I}_{N-1} + \mathbf{E}_{N-1})
\end{equation}
and thus $\det(\Sigma_{N-1}) = N^{1-(N-1)}$. Hence, using H\"{o}lder's inequality,
\begin{equation}\label{Eqn: First integral estimate}
|J^{(1)}| \leq \sqrt{N} 4\pi \left(\frac{2\pi}{N}\right)^{\frac{N-1}{2}} \E_{\theta}\bigg[ \prod_{(p,q) \in E_{\bp}} |\theta_p - \theta_q| \bigg] \leq  \sqrt{N} 4\pi \left(\frac{2\pi}{N}\right)^{\frac{N-1}{2}} \prod_{(p,q) \in E_{\bp}} \sqrt{\E_{\theta}[ (\theta_p - \theta_q)^2]},
\end{equation}
where 
\[
\E_{\theta}[f(\theta)] :=  \frac{1}{\sqrt{(2\pi)^{N-1}\det(\Sigma)}}\int_{-\infty}^{\infty} d^{N-1} \theta f(\theta) e^{-\frac{1}{2}\theta^T \Sigma^{-1} \theta}.
\]
Using the form of the covariance matrix (\ref{Eqn: Covariance matrix}) the Gaussian expectation of two random variables is $\E_{\theta}[\theta_p \theta_q] = \Sigma_{pq} = \frac{1}{N}(\delta_{pq} + 1)$. Therefore $\E_{\theta}[ (\theta_p - \theta_q)^2] = \E_{\theta}[ \theta^2_p] - 2\E[\theta_p\theta_q] + \E[\theta_q^2] = 2/N$ and so
\[
|J^{(1)}| \leq  \sqrt{N} 4\pi \left(\frac{2\pi}{N}\right)^{\frac{N-1}{2}} \left(\frac{2}{N}\right)^{k/2}.
\]

\vspace{15pt}

We now turn to the evaluation of $J^{(2)}$. Due to the condition $\sum_j n_j = N$, we have that at least one of $n_7+n_0$, $n_1 + n_2$, $n_3 + n_4$ and $n_5 + n_6$ is greater than or equal to $N/4$. Suppose this is the case for $n_3 + n_4$. Let us denote $A = A_3 \cup A_4 = [-\frac{\pi}{8},\frac{\pi}{8}]$, $B = A_2 \cup A_5 = [-\frac{\pi}{4},-\frac{\pi}{8}] \cup [\frac{\pi}{8},\frac{\pi}{4}]$ and $C = A_0 \cup A_1 \cup A_6 \cup A_7 =  [-\frac{\pi}{2},-\frac{\pi}{4}] \cup [\frac{\pi}{4},\frac{\pi}{2}]$, with $n_A = n_3 + n_4, n_B = n_2+n_5$ and $n_C = n_0 + n_1 + n_6 + n_7$ accordingly. If we write $\F := \{ \bs \in D^{(2)} : n_A \geq N/4 \}$ and account for the four possibilities of having at least $N/4$ angles in the particular segment then
\begin{equation}\label{Eqn: J2 expression}
|J^{(2)}| \leq 4 \sum_{\bs \in \F} \int_{V(\bs)} d^{N} \theta \prod_{(p,q) \in E}  |\sin(\theta_p - \theta_q)|  \prod_{(p,q) \in E^c}|\cos( \theta_p - \theta_q)| .
\end{equation}
In addition, we split $\F = \F_{>} \cup \F_{\leq}$, where for some $\epsilon >0$ we have $\F_{>} = \{\bs \in \F : n_C > N^\epsilon\}$ and $\F_{\leq} = \{\bs \in \F : n_C \leq N^\epsilon\}$ and evaluate each part separately. If $\theta_p \in A$ and $\theta_q \in C$ (or vice versa) then $|\theta_p - \theta_q| \geq \pi/8$ and so $|\cos(\theta_p - \theta_q)| \leq \cos(\pi/8) = e^{-c}$ for $c = -\log(\cos(\pi/8))>0$. In addition, for $\theta_p,\theta_q \in A \cup B$ and $\theta_p,\theta_q \in C$ we can employ the bounds (\ref{Eqn: sin and cos bounds}), and for all others write $|\sin(\theta_p - \theta_q)| \leq 1$ and $|\cos(\theta_p - \theta_q)|$. Therefore, using the arguments above for the Gaussian integral, the restriction of (\ref{Eqn: J2 expression}) to $\F_{>}$ satisfies
\begin{multline}\label{Eqn: F> expression}
\sum_{\bs \in \F_{>}} \int_{V(\bs)} d^{N} \theta \ \prod_{(p,q) \in E}| \sin(\theta_p - \theta_q)|  \prod_{(p,q) \in E} |\cos( \theta_p - \theta_q)| \\
\leq \frac{\pi^2}{4^2} \sum_{\bs \in \F_{>}} e^{-c(n_An_C - k_{AC})}  \sqrt{(n_A + n_B)n_C}\left(\frac{2\pi}{n_A + n_B}\right)^{\frac{n_A + n_B -1}{2}} \\
\times \left(\frac{2}{n_A + n_B}\right)^{\frac{k_{A\cup B}}{2}} \left(\frac{2\pi}{n_C}\right)^{\frac{n_C-1}{2}} \left(\frac{2}{n_C}\right)^{\frac{k_{C}}{2}},
\end{multline}
where $k_{AC} = \#\{(p,q) \in E : \theta_p \in A, \theta_q \in C\}$, $k_{A \cup B} =  \#\{(p,q) \in E : \theta_p, \theta_q \in A\cup B\}$ and $k_{C} =  \#\{(p,q) \in E : \theta_p, \theta_q \in C\}$. Now, given that $k_{AC} \leq k$, we have $e^{ck_{AC}} \leq e^{ck}$. Also, since $k_{A \cup B} + k_C \leq k$ and $n_A + n_B$ and $n_C$ cannot be equal to zero for $\bs \in \F$
\[
\left(\frac{2}{n_A + n_B}\right)^{\frac{k_{A\cup B}}{2}}  \left(\frac{2}{n_C}\right)^{\frac{k_{C}}{2}} \leq 2^{k/2}.
\]
In addition, $n_An_C \geq \frac{1}{4}N^{1+\epsilon}$ for $\bs \in \F_{>}$, so the expression in (\ref{Eqn: F> expression}) is less than or equal to
\begin{multline}\label{Eqn: F> expression 2}
\frac{\pi^2}{4^2}F_k e^{-\frac{c}{4}N^{1 + \epsilon}} \sum_{\bs \in \F_{>}} \sqrt{(n_A+n_B)n_C}\left(\frac{2\pi}{n_A + n_B}\right)^{\frac{n_A + n_B -1}{2}}\left(\frac{2\pi}{n_C}\right)^{\frac{n_C -1}{2}} \\
\leq \frac{\pi^2}{4^2}F_k  e^{-\frac{c}{4}N^{1 + \epsilon}} \sum_{r = N^{\epsilon}}^N \choose{N}{r} \sqrt{(N-r)r}\left(\frac{2\pi}{N-r}\right)^{\frac{N-r -1}{2}}\left(\frac{2 \pi}{r}\right)^{\frac{r-1}{2}},
\end{multline}
where $F_k = 2^{k/2}e^{ck}$ and we have used $ r= n_C$ and $N-r = n_A + n_B$. The factor $\choose{N}{r}$ accounts for the number ways of placing $r$ angles in $C$ and $N-r$ angles in $A \cup B$. The summand is maximised when $r = N/2$ and so, using the bound $\sqrt{2\pi n}n^ne^{-n} \leq n! \leq 2\sqrt{2\pi n}n^ne^{-n}$ when $n \geq 1$ for the factorial, we get the contribution from $\F_{>}$ is less than or equal to 
\begin{multline}
 (N- N^{\epsilon}) \frac{\pi^2}{4^2}F_k e^{-\frac{c}{4}N^{1 + \epsilon}}  \frac{2\sqrt{2\pi N}N^N}{(N\pi)(N/2)^N} \frac{N}{2} \left(\frac{2 \pi}{N}\right)^{\frac{N-1}{2}}\left(\frac{2 \pi}{N}\right)^{-\frac{1}{2}} 2^{N/2 - 1} \\
 = \O(N^{3/2}) \sqrt{N}\left(\frac{2 \pi}{N}\right)^{\frac{N-1}{2}} \exp\left(-\frac{c}{4}N^{1 + \epsilon} + \frac{3\ln(2)}{2}N\right),
\end{multline}
which is negligible in comparison to the contribution from $J^{(1)}$ given in (\ref{Eqn: J1 result}).

This leaves the evaluation of $\F_{\leq}$. If we restrict the expression (\ref{Eqn: J2 expression}) to $\F_{\leq}$ and follow exactly the same steps as for the contribution from $\F_{>}$ above we get that, since $n_A \geq N/4$ and $n_C \geq 1$,
\begin{multline}\label{Eqn: F<= expression}
\sum_{\bs \in \F_{\leq}} \int_{V(\bs)} d^{N} \theta \ \prod_{(p,q) \in E}| \sin(\theta_p - \theta_q)|  \prod_{(p,q) \in E} |\cos( \theta_p - \theta_q)| \\
\leq \frac{\pi^2}{4^2}F_k  e^{-\frac{c}{4}N} \sum_{r = 1}^{N^{\epsilon}} \choose{N}{r} \sqrt{(N-r)r}\left(\frac{2\pi}{N-r}\right)^{\frac{N-r -1}{2}}\left(\frac{2 \pi}{r}\right)^{\frac{r-1}{2}},
\end{multline}
which matches (\ref{Eqn: F> expression 2}), except for the exponential factor and the summation range. Now, using the bound for the factorial and removing a factor of $\sqrt{N} (2\pi/N)^{(N-1)/2}$ gives that (\ref{Eqn: F<= expression}) is less than or equal to
\[
\leq \frac{\pi^2}{4^2}F_k  e^{-\frac{c}{4}N} \sqrt{N} \left(\frac{2\pi}{N}\right)^{\frac{N-1}{2}} \sum_{r = 1}^{N^{\epsilon}}  ((1 - r/N)^{N})^{-3/2} \left(\frac{N - r}{r}\right)^{3\frac{r}{2}} (r/2\pi)^{1/2} \left(\frac{N}{N-r}\right)^{-1/2}.
\]
Finally, $r/N = \O(N^{\epsilon-1})$ we have $(1 - r/N)^N = \exp(N\ln(1 - r/N)) = \exp(-r + \O(r^2/N)) = \exp(N^{\epsilon} + \O(N^{2\epsilon-1}))$ and $((N - r)/r)^{3r/2} = \exp(3N^{\epsilon}\ln(N^{1-\epsilon} - 1)/2)$. Therefore the contribution from (\ref{Eqn: F<= expression}) is of the form
\[
\O(N^{\epsilon+1/2}) \sqrt{N} \left(\frac{2\pi}{N}\right)^{\frac{N-1}{2}}\exp\left(-\frac{c}{4}N + \frac{3}{2}N^{\epsilon} + \frac{3}{2}N^{\epsilon}\ln(N^{1-\epsilon} - 1) + \O(N^{2\epsilon-1})  \right),
\]
which, again, is negligible in comparison to (\ref{Eqn: J1 result}).
\end{proof}

\bibliographystyle{unsrt}
\bibliography{ref}

\end{document}